\documentclass{conm-p-l}
\usepackage{amsmath}
\usepackage{amssymb}
\usepackage{amsfonts}
\usepackage{mathrsfs}
\usepackage{amsthm}
\usepackage{epsfig,epic}
\usepackage[english]{babel}
\usepackage{graphics}
\usepackage{graphicx}
\usepackage{eucal}
\usepackage{amscd}
\usepackage{amssymb}
\usepackage{fancyhdr}
\usepackage[hang]{subfigure}
\usepackage{array}
\usepackage{verbatim}

\newtheorem{theorem}{Theorem}[section]
\numberwithin{equation}{section}
\newtheorem{lemma}[theorem]{Lemma}
\newtheorem{prop}[theorem]{Proposition}

\newtheorem{thm}{Theorem}
\newtheorem{cor}[theorem]{Corollary}
\newtheorem{rem}[theorem]{Remark}
\newenvironment{proofof}[2]{\begin{proof}[Proof of #1 \ref{#2}.]}{\end{proof}}
\newcommand{\be}{\begin{equation}}
\newcommand{\ee}{\end{equation}}
\newcommand{\bes}{\begin{equation*}}
\newcommand{\ees}{\end{equation*}}

\newcommand{\ud}{\mathrm{d}}
\newcommand{\BS}[3]{S_{#3}(#1, #2)}

\newcommand{\Orb}[2]{\mathcal{O}_{#2}(#1)}

\newcommand{\G}{\mathscr{G}}
\newcommand{\h}[1]{\hat{#1}}

\newcommand{\hC}[1]{\h{\mathcal{C}}{([#1]) }}

\newcommand{\st}{\, : \, }

\begin{document}
\title[A Limit Theorem for Birkhoff sums over
Rotations]{A Limit Theorem for Birkhoff Sums  of  non-Integrable Functions   over
Rotations}

\author{Yakov G. Sinai}
\address{Mathematics Department \\ Princeton University \\ Princeton \\ New Jersey 08544 \\ USA \endgraf
 Landau Institute for Theoretical Physics \\  Moscow \\ Russia}
\email{sinai@math.princeton.edu}
\thanks{The first author thanks NSF Grant DMS $0600996$ for the financial support.}

\author{Corinna Ulcigrai}
\address{School of Mathematics \\ University of Bristol \\ Bristol BS8 1TW \\ United Kingdom}
\email{corinna.ulcigrai@bristol.ac.uk}
\thanks{The second author thanks the  Clay Mathematics Institute, since part of this work was completed while she was supported by a Liftoff Fellowship.}


\dedicatory{Dedicated to M. Brin  on the occasion of his sixtieth birthday.}


\begin{abstract}
We consider Birkhoff sums of functions with a singularity of type $1/x$ over rotations and prove the following limit theorem. Let $S_N= S_N( \alpha,x)$ be the $N^{th}$ non-renormalized Birkhoff sum, where $\alpha\in [0,1)$ is the rotation number, $x\in [0,1)$ is the initial point and $(\alpha, x)$ are uniformly distributed.  
 We prove that  $S_N/N$ has a joint limiting distribution in $( \alpha,x)$ as $N$ tends to infinity. 
As a corollary, we get the existence of a limiting distribution for certain trigonometric sums. 
\end{abstract}

\maketitle

The purpose of this paper is the proof of the following theorem.
\begin{thm}\label{limitingexponentialsums}
For any Borel-measurable subset $\Omega \subset \mathbb{C}$ there exists
\be\label{complexdistr}
\lim_{N\rightarrow \infty} Leb \left\{ (\alpha, x ) \st
\frac{1}{N}\sum_{n=0}^{N-1}  \frac{1}{1- e^{2\pi i (n\alpha + x)} } \in \Omega .
\right\} = \mathbb{P}(\Omega)
\ee
Here $Leb$
denotes  the two-dimensional Lebesgue measure on
 $[0,1)\times[0,1) $ 
 and $\mathbb{P}$ is a probability measure on $\mathbb{C}$.
\end{thm}
In other words, the trigonometric sums $\frac{1}{N}\sum_{n=0}^{N-1} ({1- e^{2\pi i (n\alpha + x)} })^{-1}$ have a limiting distribution. 

The theorem follows 
as a corollary from the following more general theorem.

Let $R_{\alpha} (x) = x +\alpha $ (mod $1$) be the rotation by $\alpha \in
\mathbb{R}$ on $[0,1)$. 
Let $f(x)=f_1(x)+f_2(x)$ where
\begin{itemize}
\item[{\it(i)}]  $f_1:\mathbb{R}\backslash \mathbb{Z} \rightarrow \mathbb{R}$ is
periodic of period $1$ and $C^1$ on $\mathbb{R}\backslash \mathbb{Z}$;
\item[{\it(ii)}] $f_1(x)=\frac{c}{x}$ on  
$\left[ {\epsilon}, 0 \right) \cup  \left(0, {\epsilon}\right] $
for some
$\epsilon < 1$ and $c\ne 0$;
\item[{\it(iii)}] $f_2$
is a $1$-periodic function, which extends to a $C^1$ function on $[0,1]$.
\end{itemize}
\begin{thm}\label{limitingthm}
For any $a<b$ there exists the limit \bes
\lim_{N\rightarrow \infty} Leb \left\{ (\alpha, x  ) \st a\leq
\frac{1}{N}\sum_{n=0}^{N-1} f(R_{\alpha}^n x) \leq b \right\} =
P(a,b), \ees where $P$ is a probability measure on $\mathbb{R}$.
\end{thm}
In other words, the random variables $X_N:=\frac{1}{N}\sum_{n=0}^{N-1}
f(R_{\alpha}^n x)$ considered as functions of $\alpha$ and $x$
have a \emph{limiting distribution}.

Using periodicity of $f_1$ and redefining $f_2$ appropriately, we
can replace {\it $(ii)$} by ${\it (ii)'}$:
\begin{itemize}
\item[{\it (ii)'}] $f_1(x) = \frac{c}{x} - \frac{c}{1-x}$ for $0 < x < 1$.
\end{itemize}
In what follows we will assume that $f_1$ and $f_2$ satisfy {\it $(i)$, $(ii)'$}
and {\it $(iii)$}.

Theorem \ref{limitingexponentialsums}  follows from  Theorem
\ref{limitingthm}. Indeed, splitting into real and imaginary part,
we can write \bes
 \frac{1}{1- e^{2\pi i (n\alpha + x)} } = \frac{1}{2} + \frac{i}{2}  f( R_{\alpha}^n x) ,
\qquad f(x): = \frac{\sin 2\pi x}{1- \cos 2\pi x} .
\ees
Then $f(x)$
satisfies {\it $(i), (ii)'$ } (with $c=\frac{1}{2\pi}$)  and {\it
$(iii)$}. Hence Theorem \ref{limitingexponentialsums} is a
corollary of Theorem \ref{limitingthm}.

Let us use the notation \be\label{BS} \BS{\alpha}{x,f}{N}= \BS{\alpha}{x}{N} :=
\sum_{n=0}^{N-1} f(R^n_{\alpha} x) \ee for the $N^{th}$
non-normalized Birkhoff sum of the function $f$ under
$R_{\alpha}$. The dependence on $f$ will be omitted if there is no ambiguity. Similar theorems can be proved for expressions of
the form
\bes
\frac{1}{2N+1}\sum_{n=-N}^N  \frac{1}{e^{2\pi i (n\alpha + x)} -1 }\quad
\mathrm{ and} \quad \frac{1}{2N+1}\sum_{n=-N}^{N} f(R_{\alpha}^n x).
\ees

Another example of Birkhoff sums with this type of singularity is given by the trigonometric series of cosecants, i.e. $\sum_{n=1}^{\infty} \sin(n\pi \alpha)^{-1}$. This series was investigated by Hardy and Littlewood in \cite{HL:som}, where they prove in particular that when $\alpha$ is a quadratic irrational, the corresponding partial sums are uniformly bounded. 
     
\subsubsection*{Outline of the proof.}
The strategy of the proof is the following. For any positive
$\epsilon$ and $\delta$ we construct approximate sums
$G^{\epsilon, \delta}_N$, which are close to $\BS{\alpha}{x}{N}$
in probability, i.e. for all sufficiently large $N$~ \bes
 Leb  \{ (\alpha,x) \st \left| G^{\epsilon, \delta}_N -
\frac{1}{N}\BS{\alpha}{x}{N}\right| \geq \epsilon \} \leq \delta.
\ees
Then we prove that, for each $\epsilon$  and $\delta $,
$G^{\epsilon, \delta}_N$ has a limiting distribution as
$N\rightarrow \infty$ and  the distributions of $G^{\epsilon,
\delta}_N$ are weakly compact in $N$,  $\epsilon$  and $\delta $.
All these statements together allow to prove Theorem \ref{limitingthm}.

Our strategy is to show that $G^{\epsilon, \delta}_N$ 
 can be expressed as
 functions of quantities which do have a limiting
distribution.
 In particular, one of the quantities involved is the ratio $q_{n(N)}/N$ where
$q_n$ are denominators of the continued fraction expansion of $\alpha$ and $n(N)$ is determined by 
$q_{n(N)} \leq N < q_{n(N)+1}$. We  use the renewal-type limit theorem  proved
in \cite{SU:ren} which gives the existence of a limiting distribution for the
ratio $q_{n(N)}/N$. This theorem is recalled and generalized in
\S\ref{renewalsec}.

The other basic tool is the classical system of partitions of the
unit circle induced by the continued fraction expansion (whose
definition is recalled in \S \ref{partitionssec}). Using this
system of partitions, the Birkhoff sums in (\ref{BS}) are
decomposed onto simpler orbit segments, which we call
\emph{cycles} and analyze separately in \S\ref{cyclesec}. The key
phenomenon which implies the asymptotic behavior of the Birkhoff
sums is the cancellation between positive and negative contributions
to each cycle (see \S\ref{cancellations}) which resemble the
existence of the principal value 
 in non-absolutely converging
integrals. The decomposition into cycles is explained in
\S\ref{decompositionsec}. The proof of Theorem \ref{limitingthm}
is given in \S\ref{limitingsec}.

\vspace{1mm}
\section{Preliminaries.}\label{preliminariessec}
\subsection{Continued fractions and partitions of the
interval.}
\label{partitionssec} The following system of partitions
exists for any $R_\alpha$ with irrational $\alpha$ (see,
e.g. \cite{Si:top}). Write down the expansion of $\alpha$ as a 
continued fraction: 
\bes \alpha = [a_1, a_2, \dots, a_n,  \dots]
\ees
 and let
$\alpha_n=\frac{p_n}{q_n} = [a_1, a_2, \dots, a_n] $ be the
$n^{th}$ approximant. Let $\{x \}$ be the fractional part of $x$.
Denote by 
\bes
\begin{array}{ll}
\Delta^{(n)}:= \Delta^{(n)}_0 = &
\left\{
\begin{array}{ll} \left[ 0,\{ q_n \alpha \} \right) & \qquad \mathrm{if} \,\,  n
\,\,  \mathrm{is}\,\,  \mathrm{even};
\\
 \left[  \{ q_n \alpha \} ,1  \right) & \qquad \mathrm{if} \,\,  n \,\,
\mathrm{is}\,\,  \mathrm{odd}.
\end{array}
\right.
\end{array}
\ees
\begin{figure}
\centering
\includegraphics[width=0.95\textwidth]{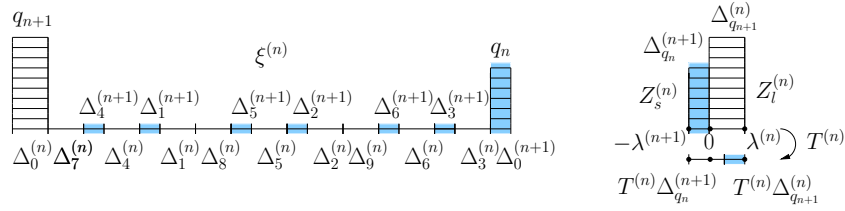}
\caption{The partition $\xi^{(n)}$ \label{torri}
and its representation into towers $Z^{(n)}_{l}$ and $Z^{(n)}_{s}$.}
\end{figure} 
For $n$ even, the intervals $\Delta^{(n)}$ and
$\Delta^{(n+1)}$ are left-most and right-most subintervals of
$[0,1)$, with endpoints  $0$ and $1$ respectively (see Figure \ref{torri}, left). 
Put \bes
 \Delta^{(n)}_j := R_{\alpha}^j \Delta^{(n)}_0 .
\ees
Denote by $\lambda^{(n)}$ the length of $\Delta^{(n)}$. 
Clearly $\lambda^{(n)}$ is also the length of any interval $\Delta^{(n)}_j$.

For any $n$, the intervals $ \Delta^{(n)}_j $, $0\leq j < q_{n+1}$
and $ \Delta^{(n+1)}_j $, $0\leq j < q_{n}$ are pair-wise disjoint and their
union is the whole interval $[0,1)$ (see Figure \ref{torri}, left). Denote by $\xi^{(n)}$ the
partition of $[0,1)$ into the intervals $ \Delta^{(n)}_j $ with $0\leq
j < q_{n+1}$ and $ \Delta^{(n+1)}_j $ with $0\leq j <
q_{n}$.  Then $\xi^{(n+1)}\geq \xi^{(n)}$ in the sense of partitions.  

Consider the union $\Delta(n) := \Delta^{(n)}\cup \Delta^{(n+1)}$.
The set $\Delta(n)$, which, as a subset of $[0,1)$, is the union of
two intervals, can be considered ($\mathrm{mod}\, 1$) as a
subinterval of the unit circle $S^1$, with endpoints on the
opposite sides of $0$, i.~e.~when $n$ is even, $\Delta(n) = [-\lambda^{(n+1)}, \lambda^{(n)})$ (see Figure \ref{torri}, right). Consider the induced map $T^{(n)}$ obtained
as the first return map of $R_{\alpha}$ on  $\Delta(n)$. Then
$T^{(n)}$ is an exchange of the two intervals $\Delta^{(n+1)}$ and
$\Delta^{(n)}$. More precisely, if $n$ is even, then
\bes
T^{(n)}(x) = \left\{ \begin{array}{lcl} x - \lambda^{(n+1)} &(\mathrm{mod}\, 1)&
\mathrm{if}\,\, x\in  \Delta^{(n)} \\  x + \lambda^{(n)}  & (\mathrm{mod}\, 1)
& \mathrm{if}\,\,  x\in  \Delta^{(n+1)} \end{array} \right.
\ees
 and similarly for  odd $n$.

Assume  $n$ is even. The intervals $\Delta^{(n)}_j$ and
$\Delta^{(n+1)}_j$ can be represented as floors of two towers, on the 
top of  $\Delta^{(n)}_0$ and $\Delta^{(n+1)}_0$ respectively,
where $j$ increases with the height of the floor in the tower, as in Figure \ref{torri}, left.
Hence the number of floors in the two towers are $q_{n+1}$ and
$q_n$ respectively. Let us denote the two towers by\footnote{The
subscripts \emph{l} and \emph{s} stay for \emph{large} and
\emph{small} respectively, since the tower $Z^{(n)}_{l}$ is both
larger and taller than $ Z^{(n)}_{s}$.} \bes Z^{(n)}_{l} =
\cup_{j=0}^{q_{n+1}-1}\Delta^{(n)}_j ; \quad Z^{(n)}_{s} =
\cup_{j=0}^{q_{n}-1}\Delta^{(n+1)}_j . \ees Under the action of
$R_{\alpha}$  each point not in the last floor (i.e.~not in
$\Delta^{(n)}_{q_{n+1}-1}$ or $\Delta^{(n+1)}_{q_{n}-1}$) moves
vertically upwards to the next floor. The action on the last floor
is determined by $T^{(n)}$: if e.g.~$x\in
\Delta^{(n)}_{q_{n+1}-1}$ and $x=R_{\alpha}^{q_{n+1}- 1}y$ then 
$R_{\alpha}x = T^{(n)}y$.

\subsubsection{Recursive structure of the partitions.}
Let us also recall how to construct $\xi^{(n)}$ inductively. Given
$\xi^{(n)}$, the partition  $\xi^{(n+1)}$ is obtained from
$\xi^{(n)}$ as follows: the intervals $\Delta^{(n+1)}_j$, $0\leq j
< q_{n}$ are also elements of the partition $\xi^{(n+1)}$. Each
$\Delta^{(n)}_j$ is decomposed in $a_{n+2}+1$ subintervals, more
precisely in $a_{n+2}$ intervals of length $\lambda^{(n+1)}$ and
a reminder, which is $\Delta^{(n+2)}_j$ (see for example Figure \ref{3refinements}). If $n$ is even, the
reminder is the left-most interval of $\Delta^{(n)}_j$, while the
other intervals, from left to right, are   $\Delta^{(n+1)}_{q_n+ j+ i
q_{n+1}}$ with $i= 0, \dots, a_{n+2}-1$ (as in Figure \ref{torri}, left). Hence, we have the following remark. 
\begin{rem}\label{separation}
Each pair of intervals of  $\xi^{(n)}$ both belonging to the tower $Z^{(n)}_{s}$ are separated by $a_{n+1}$ partition elements belonging to $Z^{(n)}_{l}$.
\end{rem}

Given $m< n$, consider  $\Delta^{(m)}_j \in   \xi^{(m)}$ . Since  $\xi^{(n)}>
\xi^{(m)}$, $\Delta^{(m)}_j$ is partitioned into elements of $\xi^{(n)}$.
Analyzing the recursive construction of the partitions
$\xi^{(n)}$, we have the following.
\begin{rem}\label{recursivepartitions}
The partition of $\Delta^{(m)}_j$ into elements of  $\xi^{(n)}$ is completely
determined by $\lambda^{(n)}$,  $\lambda^{(n+1)}$ and $a_{n-k+2}$, $k= 0, \dots,
n-m$.
\end{rem}

\subsection{The renewal-type limit theorem for denominators.}\label{renewalsec}
The existence of the limiting distribution relies on the following limit
theorem. Let  $p_n/q_n$ be the approximants of   $\alpha  = [a_1, a_2, \dots]$ and
$q_n=q_n(\alpha)$ the corresponding denominators as functions of $\alpha$.
\begin{thm}[\cite{SU:ren}]\label{main}
Given $N>0$, introduce 
\be \label{n(N)def} n(N) = n(N,\alpha) = \min \{ n \in
\mathbb{N} \, |  \, \,  q_n > N \,\, \mathrm{and } \,\,  n \, \,
\mathrm{is} \,\,\mathrm{even}  \} . \ee Fix also an integer
$M\geq0$. Then the ratio $\frac{q_{n(N)}}{N}$ and the entries
$a_{n(N)+k}$ for $| k | \leq  M$
 have a joint limiting distribution, as $N$ tends to infinity, with respect to the uniform distribution on $\alpha$. 
\end{thm}

Theorem \ref{main} means that for each $M\geq0$ there
exists a probability measure $P_M$ on $(1,\infty) \times {\mathbb{N}_+}^{2M+1}$
such that for all $a, b >1$ and $c_k\in \mathbb{N}_+$ with $| k|\leq  M$, 
\be \label{limitingdenom}
\begin{split}
\lim_{N\rightarrow \infty}  Leb  \left\{ \alpha  \st a < \frac{q_{n(N)}(\alpha)}{N} <
b,  \quad a_{n(N)+ k} = c_k, \, | k|\leq  M \right\} =
\\
P_M \left(
(a,b), c_{-M}, \dots, c_{M} \right).
\end{split}
\ee
Theorem \ref{main} is a slight modification of Theorem $1$, \cite{SU:ren}. The
differences and a sketch on how to modify the proof of Theorem $1$ in
\cite{SU:ren} to obtain Theorem \ref{main} are pointed out in the Appendix \S\ref{appendixthmdiff}.

As a corollary of Theorem \ref{main}, we have the following.
\begin{cor}\label{limitingquantities}
The quantities
\bes \frac{q^{(n(N))}}{q^{(n(N)+1)}}, \quad \frac{\lambda^{(n(N)+1)}}{\lambda^{(n(N))}}, \quad \frac{1}{q^{(n(N))}\lambda^{(n(N)+1)}}, \quad \frac{1}{q^{(n(N)+1)}\lambda^{(n(N))}}
\ees
 have a limiting distribution as $N$
tends to infinity.
\end{cor}
\begin{proof}
 Let us recall that $q_n$ and $\lambda^{(n)}$
satisfy the following recurrent relations (see \cite{Kh:con} and \cite{Si:top}
respectively):
\be\label{recursiverelations}
q_{n+1}= a_{n+1} q_n + q_{n-1}, \qquad \lambda^{(n-1)}= a_{n+1} \lambda^{(n)} +
\lambda^{(n+1)}, \qquad n\geq 1.
\ee
Using them inductively (see \cite{Kh:con} or \cite{SU:ren}), it is easy to show
that
\bes
\frac{q_{n}}{q_{n+1}} = [a_{n+1}, a_{n}, \dots, a_1] , \qquad
\frac{\lambda^{(n+1)}}{\lambda^{(n)}} = [a_{n+2}, a_{n+3}, \dots] .
\ees

Moreover, reasoning as in \cite{SU:ren}, we also have 
\begin{eqnarray}
&&\left|\frac{q_{n}}{q_{n+1}} - [a_{n+1}, a_{n}, \dots,
a_{n-K}]\right|\leq \frac{1}{{2}^{\frac{K+1}{2}}}; \nonumber \\ && \nonumber \left|
\frac{\lambda^{(n+1)}}{\lambda^{(n)}} - [a_{n+2}, a_{n+3}, \dots, a_{n+K}]
\right| \leq \frac{1}{{2}^{\frac{K-2}{2}}}, \end{eqnarray}
where the
exponential convergence is uniform in $\alpha$. Hence, since by
Theorem \ref{main}, for each $K$,  $a_{n(N)}$,  $a_{n(N)\pm 1},
\dots, a_{n(N)\pm K}$ have a joint limiting distribution as $N$
tends to infinity,
  $\frac{q_{n(N)}}{q_{n(N)+1}}$ and
$\frac{\lambda^{(n(N)+1)}}{\lambda^{(n(N))}}$ also have a limiting
distribution.

For the last two quantities, recall, e.g.~from \cite{Kh:con}, that
\bes
\lambda^{(n)} = \left|q_n \alpha  - p_n \right| = \frac{1}{q_{n+1} + q_n
\alpha_{n+1}}, \quad \mathrm{where} \, \,  \alpha_{n+1} = \G^{n+1}\alpha = [a_{n+2},
a_{n+3},  \dots ].
\ees
Hence, in particular 
\be\label{towersareas}
 \frac{1}{2} \leq \lambda^{(n)} q_{n+1}  \leq 1, \qquad \lambda^{(n+1)} q_{n}\leq 1.
\ee
Moreover, since
\bes
\frac{1}{q_{n(N)+1}\lambda^{(n(N))}} = 1 + \frac{q_{n(N)}}{q_{n(N)+1}}
[a_{n(N)+2}, a_{n(N)+3}, \dots]
\ees
the ratio
$\frac{1}{q_{n(N)+1}\lambda^{(n(N))}}$ and similarly 
$\frac{1}{q_{n(N)}\lambda^{(n(N)+1)}}$  have limiting distributions.
\end{proof}
\section{Analysis of a cycle.}\label{cyclesec}

In this section and in \S\ref{decompositionsec}, we consider only Birkhoff sums of the function $f_1$. Since $f_2$ is
integrable, Birkhoff sums of $f_2$ are easily  controlled  in
\S\ref{limitingsec} with the help of Birkhoff ergodic theorem.

We first investigate in this section a special type of Birkhoff sum, which is used in 
\S\ref{decompositionsec} as a building block to decompose any other Birkhoff
sum. Assume that ${x}\in \Delta^{(n)}$ and $q=q^{(n+1)}$ if $n$ is even or
$q=q^{(n-1)}$ if $n$ is odd and consider the Birkhoff sum $\BS{\alpha}{x}{q}$.
We call the orbit segment $\{ R_{\alpha}^i {x}, \, i=0, \dots, q-1\}$ a
\emph{cycle} and  $\BS{\alpha}{x}{q}$ is a \emph{sum over a cycle}.
We remark that all points of a cycle are
contained in  the same tower and  there is exactly one point in each floor of the tower; for this
reason, we sometimes refer to $\BS{\alpha}{x}{q}$ as a \emph{sum over a tower}
(see also \cite{Ul:mix}). In section \S\ref{decompositionsec} we will refer to
$n$ as the \emph{order} of the cycle. 

To simplify the analysis, we assume in what follows that $n$ is even
and consider only the partitions  $\xi^{(n)}$ with $n$ even and their cycles.
The following proposition shows that the value of a sum over a cycle is
determined essentially by the closest point to the endpoint.
\begin{prop}\label{cycleprop} Let  $\BS{\alpha}{x}{q}$ be a sum along a cycle,
$q=q_{n+1}$ if $x\in \Delta^{(n)}$ or $q=q_{n}$ if  $x\in \Delta^{(n+1)}$.
For each $\epsilon>0$ there exist $K=K(\epsilon)$ and
functions $g_n^{\epsilon}(\alpha, {x}, q)$, $n\in \mathbb{N}$, which depend
only on the following quantities
\be\label{gdependence}
g_n^{\epsilon}  (\alpha, {x}, q) = \left\{ \begin{array}{ll} g_n^{\epsilon}
\left( \frac{{x}}{\lambda^{(n)}},  \frac{\lambda^{(n+1)}}{\lambda^{(n)} },
\frac{1}{q_{n+1} \lambda^{(n)} },  a_{n+2}, a_{n+1},
 \dots, a_{n- K} \right) 
& \, \,   \mathrm{if}\, {x}\in  \Delta^{(n)} 
\\g_n^{\epsilon}
\left( \frac{{x}}{\lambda^{(n+1)}},  \frac{\lambda^{(n+1)}}{\lambda^{(n)} },
\frac{1}{q_{n} \lambda^{(n+1)} },  a_{n+2}, a_{n+1},
 \dots, a_{n- K} \right) 
& \, \, \mathrm{if}\, {x}\in  \Delta^{(n+1)} 
\end{array} \right.
\ee
and such that,
letting $q=q^{(n+1)}$ if $x\in \Delta^{(n)}$ or $q=q_{n}$ if  $x\in \Delta^{(n+1)}$, 
we have
\be\label{gapproximates}
\left| \frac{1}{q}  \BS{\alpha}{{x}}{q} - g_n^{\epsilon} (\alpha, {x}, q)
\right| \leq 
 \epsilon.
\ee
\end{prop}
The proof of Proposition \ref{cycleprop} is given in \S\ref{cancellations}. The
key ingredient which allows to reduce the sum over a cycle to finitely many terms
(and hence to an expression given by  $g_\epsilon$ depending on the above
variables) is that there are cancellations between the two sides, positive and
negative, of the singularity. The cancellations occur because the sequence of
closest points to $0$ is given by a rigid translate of the sequence of closest
points to $1$ (see Corollary \ref{translatesxjyj}). In order to prove this fact,
we first show, in \S\ref{almostsymmetry}. that the partitions $\xi^{(n)}$ have a
property of almost symmetry (see Lemma \ref{almostsymmetrylemma}, in %
\S\ref{almostsymmetry}).

\subsection{ Almost symmetry of  the partitions.}\label{almostsymmetry}
Consider the partition $\xi^{(n)}$
and let $z^{(n)}_i$, for $i=0,\dots, q_{n+1}-1$ denote the middle points of the
intervals $\Delta^{(n)}_j$, $0\leq j < q_{n+1}$, rearranged in increasing order,
so that  $z^{(n)}_0< z^{(n)}_1<\dots < z^{(n)}_{q_{n+1}-1}$ and similarly let
$z^{(n+1)}_i$, for $i=0,\dots, q_{n}-1$ be the middle points of the intervals
$\Delta^{(n+1)}_j$, $0\leq j < q_{n}$ rearranged in increasing order (see Figure \ref{zdisplacement}).

Since we are interested in comparing the functions $\frac{1}{x}$ and
$\frac{1}{1-x}$ evaluated along orbits segments which contain a point inside
each of these intervals, we want to understand what happens to the middle points
under the reflection $x\mapsto \sigma(x):=(1-x)$.

Consider the set of reflected  points  $\{ 1- z^{(n)}_i$, for $i=0,\dots,
q_{n+1}-1\}$ and  let  $z'^{(n)}_i$, for $i=0,\dots, q_{n+1}-1$ denote its
elements rearranged in increasing order. Similarly, let  $z'^{(n+1)}_i$, for
$i=0,\dots, q_{n}-1$  be the monotonical rearrangements of the points  $\{ 1-
z^{(n+1)}_i$, for $i=0,\dots, q_{n}-1\}$.

\begin{figure}
\centering
\includegraphics[width=0.95\textwidth]{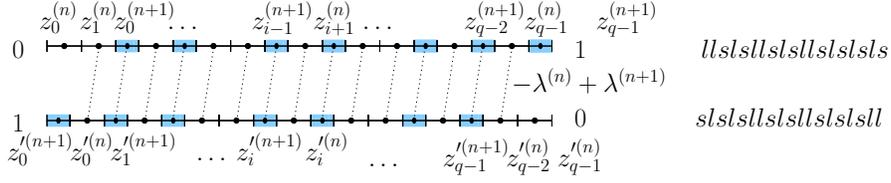}
\caption{An example of the relations (\ref{translatesn}, \ref{translatesn+1}) between  $z^{(n)}_i $, $z^{(n+1)}_i $ and $z'^{(n)}_i $, $z'^{(n+1)}_i $.
\label{zdisplacement}}
\end{figure} 
\begin{lemma}\label{translatesz} Let $n$ be even.
The two sequences given by the points $z^{(n)}_i $ and the points $z^{(n+1)}_i $
respectively, excluding the closest point to $0$, i.e.~$z^{(n)}_0$, and the
closest point to $1$, i.e.~$z^{(n+1)}_{q_n-1}$, are rigid translates of each
other, i.e. they satisfy:
\begin{eqnarray}
z'^{(n)}_i  &=& z^{(n)}_{i+1} + \lambda^{(n+1)} -  \lambda^{(n)}; \qquad i=0,
\dots, q_{n+1}-2 \label{translatesn}\\ z'^{(n+1)}_i  &=& z^{(n+1)}_{i-1}  +
\lambda^{(n+1)} - \lambda^{(n)} , \qquad  i=1, \dots, q_{n}-
1.\label{translatesn+1}
\end{eqnarray}
\end{lemma}
The restriction on the parity simplify the number of cases in the statement, but
similar properties could be proved for $n$ odd.

Lemma \ref{translatesz} will follow as a corollary of an almost-symmetry
property of the partitions $\xi^{(n)}$ (Lemma \ref{almostsymmetrylemma} below).
Let us consider the following coding of the partitions $\xi^{(n)}$. The unit
interval $[0,1)$ is decomposed into $q_n+q_{n+1}$ subintervals which are elements of
the partition $\xi^{(n)}$ and either belong to $Z^{(n)}_l$ (i.e.~are of the form
$\Delta^{(n)}_j$ for some $0\leq j < q_{n+1}$) or to $Z^{(n)}_s$ (i.e.~are of
the form $\Delta^{(n+1)}_j$ for some $0\leq j < q_{n}$). We will call them
intervals of type $l$ and type $s$ respectively (large or short). Let
$\underline{\omega}^{(n)}= \omega^{(n)}_1 \cdots \omega^{(n)}_{q_n+q_{n+1}}$ be a
string of letters $l$ and $s$, where $\omega^{(n)}_i=l $ or $\omega^{(n)}_i=s $
according to the type of the $i^{th}$ interval of $\xi^{(n)}$  (where intervals of the partition are ordered from left to right in $[0,1)$). 
For example, the string coding the partition $\xi^{(n)}$ in Figure \ref{zdisplacement} (which is the same that appears also in Figure \ref{torri}) is $\underline{\omega}^{(n)}=llslsllslsllslslsls$.

Let $\underline{\omega}'^{(n)}= \omega'^{(n)}_{q_{n+1}}, \dots, \omega'^{(n)}_1$
be the \emph{reflected string}, which encodes the type of intervals after the
reflection $x\mapsto 1-x$.  Then, the following almost-symmetry property is
satisfied by the partitions $\xi^{(n)}$.
\begin{lemma}[almost symmetry of $\xi^{(n)}$]\label{almostsymmetrylemma}
For all $n$, all the letters of the strings $\underline{\omega}^{(n)}$ and
$\underline{\omega}'^{(n)}$ coincide with the exception of the first and last,
i.e.~
\be \label{equalities}
\omega'^{(n)}_{i} = \omega^{(n)}_{i}\quad \mathrm{for\, all }\quad  2\leq i \leq
q_{n}+q_{n+1}-1.
\ee
More precisely, 
\begin{eqnarray}&& \mathit{if}\,  n\,  \mathit{is \, even}, \, \, 
\underline{\omega}^{(n)} = (l \omega^{(n)}_2 \cdots \omega^{(n)}_{q_{n}+q_{n+1}-1}
s)\,\, \mathit{and} \,\,  \underline{\omega}'^{(n)} = (s \omega^{(n)}_2 \cdots
\omega^{(n)}_{q_{n}+q_{n+1}-1} l);
\nonumber \\ && \mathit{if}\,  n\,  \mathit{is \, odd}, \, \,
\underline{\omega}^{(n)} = (s \omega^{(n)}_2 \cdots
\omega^{(n)}_{q_{n}+q_{n+1}-1} l)\, \, \mathit{and}\,\,  \underline{\omega}'^{(n)} = (l \omega^{(n)}_2
\cdots \omega^{(n)}_{q_{n}+q_{n+1}-1} s). \nonumber
\end{eqnarray}
Moreover, for $n\geq 4$, $\omega^{(n)}_{2}={\omega}'^{(n)}_2= \omega^{(n)}_{q_{n}+q_{n+1}-1}= {\omega}'^{(n)}_{q_{n}+q_{n+1}-1}=l$.
\end{lemma}
\begin{proof}

The proof proceeds by induction on $n$. For $n=0$, $\underline{\omega}^{(0)}=
(ll\cdots ls)$ where the number of occurrences of $l$ is given by $a_1$. Hence,
$\underline{\omega}'^{(0)}= (sll\cdots l)$ and there is nothing to prove. Assume
that the almost-symmetry is proved for $\xi^{(n)}$ ($\omega'_{i} = \omega_{i}$ for
all $2\leq i \leq q_n+q_{n+1}-1$) and $n$ is even. As it can be seen easily analyzing the recursive
construction of $\xi^{(n)}$ in \S\ref{partitionssec}, the new string
$\underline{\omega}^{(n+1)}$ is obtained from   $\underline{\omega}^{(n)}$ by
substituting  each letter $s$ with $l$ (since $\lambda^{(n+1)}$ which was the
shortest length in $\xi^{(n)}$ is now the longest one in $\xi^{(n+1)}$) and
substituting each letter $l$ with $sll\dots l$ where the number of occurrences
$l$ is given by $a_{n+1}$, see for example Figure \ref{3refinements}.
 \begin{figure}
\centering
\includegraphics[width=0.9\textwidth]{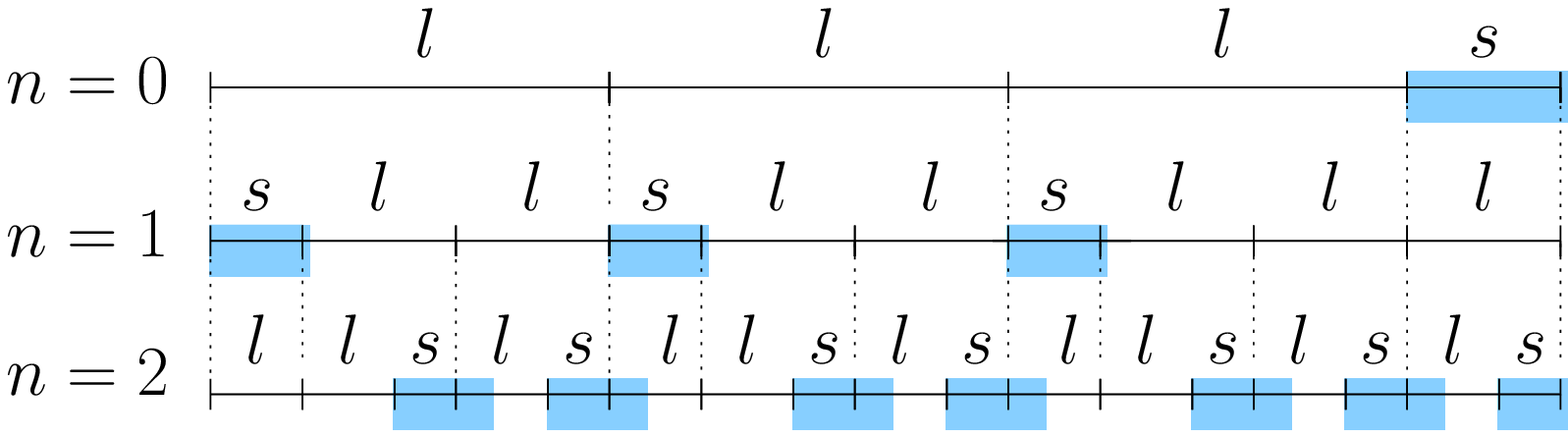}
\caption{An example of partitions $\xi^{(0)}$,  $\xi^{(1)}$,  $\xi^{(1)}$ (where $a_1=3$, $a_2=2$, $a_3=1$).\label{3refinements}}
\end{figure}

To verify the desired identities on the letters in $\underline{\omega}^{(n+1)}$ and
$\underline{\omega}'^{(n+1)}$ it is enough to verify that the letters $s$ occur
in the same positions (with the exception the first and last letter of the string). Let
$l(i)$ denote the number of letters $l$ among $\omega^{(n)}_j$ with $0\leq j <
i$ (i.e.~ the cardinality of $\omega^{(n)}_j=l$ with  $0\leq j < i$). Since all
$s$ in  $\underline{\omega}^{(n)}$ become $l$, the only $s$ in the string
$\underline{\omega}^{(n+1)}$ appear inside each block $sll\dots l$. Moreover,
each occurrence of $l$ in $\underline{\omega}^{(n)}$ generates a string of
length $a_{n+1}+1$ in $\underline{\omega}^{(n+1)}$. Hence ${\omega}_j^{(n+1)}= s$
iff ${\omega}_i^{(n)}= l$ and $j= i+  a_{n+1}l(i)$ (for $1\leq i \leq q_{n}+q_{n+1}$).

Similarly the string $\underline{\omega}'^{(n+1)}$ is obtained from
$\underline{\omega}'^{(n)}$ by substituting $s$ with $l$ and substituting each
symbol $l$ with $ll\dots ls$ ($a_{n+1}$ copies of $l$). If $l'(i)$ denote the
number of letters $l$ (i.e.~$\omega'^{(n)}_j=l$) among $\omega'^{(n)}_j$ with
$0\leq j < i$, then ${\omega'}_{j'}^{(n)}=s $ iff ${\omega}_i^{(n)}= l$ and 
$j'= i+ a_{n+1}(l'(i)+1)$ (for $i=1,\dots, q_{n}+q_{n+1}$). By the inductive
assumption, since $\omega^{(n)}_i$ and $\omega'^{(n)}_i$  coincide for all
$i\neq 1, i\neq q_{n}+q_{n+1}$ but $\omega'^{(n)}_1=s$, we have $l'(i) = l(i)-1$.
Hence, for $2\leq i \leq q_{n-1}+q_{n-2}-1$, we have ${\omega}_i^{(n)}= l$ iff
${\omega'}_i^{(n)}= l$ and
 $j'= i+ a_{n+1}(l'(i)+1) = i+ a_{n+1}l(i) = j $, which implies, as we wanted,
that ${\omega'}_{j'}^{(n+1)}=s $ iff   ${\omega}_j^{(n+1)}= s$. The proof for
 odd $n$ is analogous.

From the definition of $\Delta^{(n)}$ we have immediately that $\omega^{(n)}_1=l$, $\omega^{(n)}_{q_{n-1}+q_{n-2}}=s$  for $n$ even and $\omega^{(n)}_1=s$, $\omega^{(n)}_{q_{n-1}+q_{n-2}}=l$ for $n$ odd and the last equalities follow from the (\ref{equalities}) and the fact that 
 two $s$ are never nearby.
\end{proof}
\begin{proofof}{Lemma}{translatesz}
Assume $n$ is even. 
The points $z_i^{(n)}$ and ${z'}_i^{(n)}$  for $0\leq i< q_{n+1}$ are
middle points of intervals of type $l$ respectively before and after the
reflection. Let us first prove (\ref{translatesn}) for $i=0$. 
 The first interval of the partition $\xi^{(n)}$ is of type
$l$ and hence contains $z_0^{(n)}$, while ${z'}_0^{(n)}$ belongs to the second
interval after the reflection, since $\omega'^{(n)}_0=s$. Unless
the string has the length $2$ and is $ls$ (in which case there is nothing to prove),  by Lemma \ref{almostsymmetrylemma}
also  $\omega^{(n)}_1=l$, so the
first two intervals are both of type $l$ and  we have $ {z'}_0^{(n)} = {z}_1^{(n)}
- \lambda^{(n)} + \lambda^{(n+1)}$. Moreover, since by Lemma
\ref{almostsymmetrylemma}, the strings $\underline{\omega}^{(n)}$ and
$\underline{\omega}'^{(n)}$ coincide after the first element, also $
{z'}_i^{(n)} = {z}_{i+1}^{(n)} - \lambda^{(n)} + \lambda^{(n+1)}$ for all
$i=0,\dots, q_{n+1}-2 $ (see Figure \ref{zdisplacement}).

Similarly, the points $z_i^{(n+1)}$ and ${z'}_i^{(n+1)}$  for $0\leq i< q_{n}$
are middle points of intervals of type $s$. In this case, ${z'}_0^{(n+1)}$
belongs to the first interval after the reflection ($\omega'^{(n)}_0=s$) and has
to be kept aside, while $z_0^{(n+1)}$ and ${z'}_1^{(n+1)}$ belong respectively to the
$(a_{n+2}+1)^{th}$ interval before reflection and to the $(a_{n+2}+1)^{th}$ after reflection. Since the strings
are ${\omega}^{(n)}_{0}{\omega}^{(n)}_{1}\cdots {\omega}^{(n)}_{a_{n+1}} =
ll\cdots l s$ and ${\omega'}^{(n)}_{0}{\omega'}^{(n)}_{1}\cdots
{\omega'}^{(n)}_{a_{n+1}+2} =sl\cdots ls$ respectively, ${z'}_1^{(n+1)}=
z_0^{(n+1)}  -(a_{n+2}+1) \lambda^{(n)} +a_{n+2} \lambda^{(n)} + \lambda^{(n+1)} = z_0^{(n+1)}  - \lambda^{(n)} + \lambda^{(n+1)}$ and, again by  Lemma
\ref{almostsymmetrylemma}, since the strings then coincide, also  $
{z'}_{i}^{(n+1)} = {z}_{i-1}^{(n)} - \lambda^{(n)} + \lambda^{(n+1)}$ for all
$i=1,\dots, q_{n}-1 $ (see again Figure \ref{zdisplacement}).
\end{proofof}

\subsection{Cancellations.}\label{cancellations}
\begin{figure}
\centering
\includegraphics[width=0.9\textwidth]{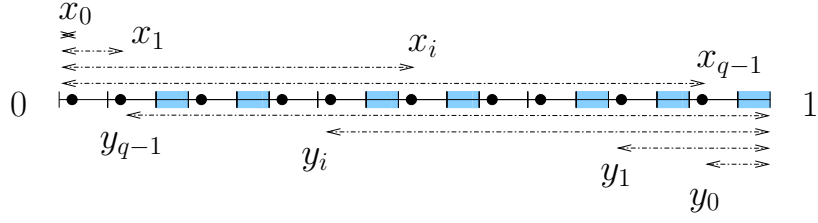}
\caption{The distances $x_i$ and $y_i$, $0\leq i < q$, from $0$ and $1$ respectively ($n$ even).\label{xandy}}
\end{figure}
Let $n$ be even, ${x}\in \Delta{(n)}$ and let $q=q_{n}$ or $q_{n+1}$ according to whether
${x}\in \Delta^{(n+1)}$ or ${x}\in \Delta^{(n)}$.  Consider the orbit cycle $\{R_{\alpha}^i {x}
\st i=0, \dots, q-1\}$, which is an orbit along a tower of $\xi^{(n)}$. Let us rename the points of  $\{R_{\alpha}^i {x}\}_{i=0}^{q-1}$ in
increasing order, so that
\bes 0<{x}_0 <x_1< \dots < x_{q-1}<1, \qquad \bigcup_{i=0}^{q-1} \{ R_{\alpha}^i
{x} \} = \bigcup_{i=0}^{q-1} \{ x_i \}.
\ees
Similarly rearrange in increasing order distances from $1$, i.~e.~the elements of
$\{1-R_{\alpha}^i {x}\}_{i=0}^{q-1}$, 
 renaming them by
\bes 0<y_0 < y_1 < \dots < y_{q-1}<1,   \qquad \bigcup_{i=0}^{q-1} \{ 1-
R_{\alpha}^i {x} \} = \bigcup_{i=0}^{q-1} \{ y_i \}.
\ees
From the structure of the partitions described in the second part of Lemma \ref{almostsymmetrylemma}, one can easily check the following (see also Figure \ref{displacementfig}).
\begin{rem} \label{firstlastpoints}
If $x\in \Delta^{(n)}$, $x=x_0$ and $y_0= \lambda^{(n+1)} +\lambda^{(n)} - x_0$,
while if $x\in \Delta^{(n+1)}$, $y_0=1-x$ and $x_0= \lambda^{(n)}a_{n+1}
+\lambda^{(n+1)} - y_0$.
\end{rem}
For the other points, from the partition almost-symmetry expressed by Lemma
\ref{translatesz}, we have the following (see an illustration in Figure \ref{xandy}).
\begin{cor}\label{translatesxjyj}
For all $1\leq j \leq q$
\begin{eqnarray}
y_i  - x_{i+1} &=&  \lambda^{(n+1)} -  2 x_0, \qquad  i=0,\dots, q_{n+1}-2 ,
\qquad  \mathrm{if} \, x\,  \in \Delta^{(n)};\label{differencen}\\
y_i  - x_{i-1} &=&  2 y_0 -  \lambda^{(n)}
 ,  \qquad  i=1, \dots,  q_{n}-1 \qquad \mathrm{if}\,  x \in\,
\Delta^{(n+1)}.\label{differencen+1}
\end{eqnarray}
\end{cor}
\begin{figure}
\begin{center}
\subfigure[$x\in \Delta^{(n)}$, $q=q_{n+1}$]{\label{xdisplacementn}
\includegraphics[width=0.89\textwidth]{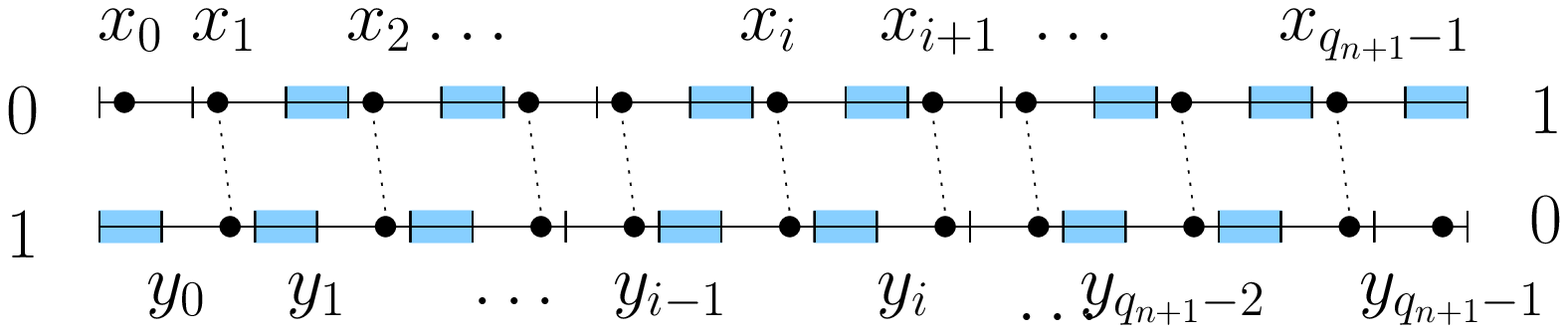}}
\end{center}
\begin{center}
\subfigure[$x\in \Delta^{(n+1)}$, $q=q_{n}$]{\label{xdisplacementn+1}
\includegraphics[width=0.80\textwidth]{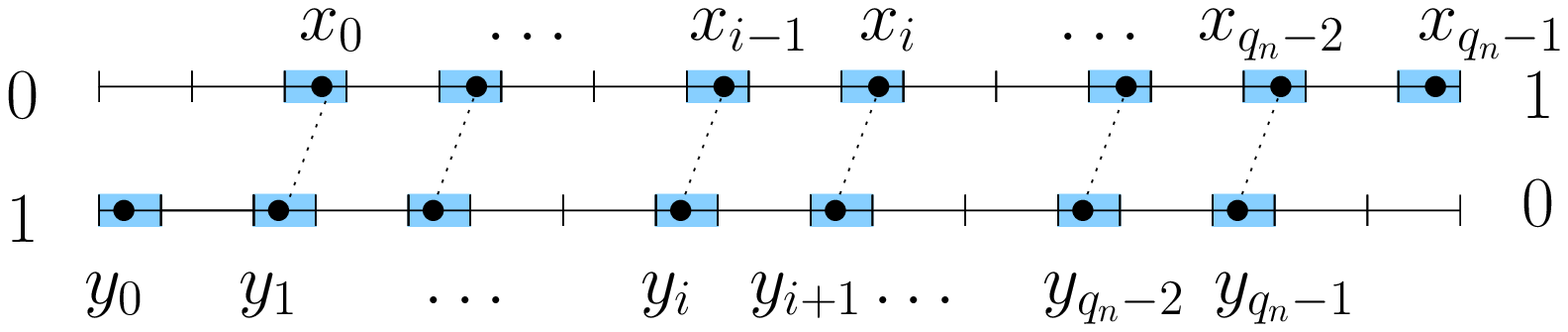}}
\end{center}
\caption{\label{displacementfig}An example of the relations (\ref{differencen}, \ref{differencen+1}) between $x_i$ and $y_i$
.}
\end{figure}
\begin{proof}
Assume that ${x}\in \Delta^{(n)}$ (see Figure \ref{xdisplacementn}).
Since, for some $0\leq k < q_{n+1}$, $x_i$ and $z_i^{(n)}$ both belong to the same $\Delta^{(n)}_k$, which is a
rigid translate of $\Delta^{(n)}_0$, 
  we have $x_i - z_i^{(n)} = x_0 - z_0^{(n)} = x_0 -  \lambda^{(n)}/2$.
Similarly, both $y_{j}$ and ${z'}_{j}^{(n)}$ belong to the same
$\Delta^{(n)}_{k'}$, which is a rigid translate and a reflection of
$\Delta^{(n)}_0$, hence
 $y_{j} - {z'}_{j}^{(n)} =  \lambda^{(n)}/2 - x_0$. Thus,  $y_{j}-x_i =
{z'}_{j}^{(n)} -  z_{i}^{(n)} +   \lambda^{(n)} - 2 x_0$. Using this relation,
(\ref{differencen}) follows from (\ref{translatesn}) in Lemma \ref{translatesz}.
The argument to prove (\ref{differencen+1}) when ${x}\in \Delta^{(n+1)}$ is
analogous (see Figure \ref{xdisplacementn+1}) and reduces to (\ref{translatesn+1}) in Lemma \ref{translatesz}.
\end{proof}
We remark that the points of $\{R_{\alpha}^i {x}, \, i=0, \dots, q-1\}$
belong to  different floors of the tower of the partition  $\xi^{(n)}$ and are
in the same relative position inside them. 
Hence, we have the following.
\begin{rem}\label{minimumdistance}
The minimum distance $\min_{i\neq j} |x_i -x_j|$ is bounded below by
$\lambda^{(n)}$. In particular, $x_j \geq j\lambda^{(n)}$ and similarly $y_j \geq j\lambda^{(n)}$ for $1\leq j \leq q_{n+1}-1$. 
\end{rem}
Moreover,  by Remark \ref{separation}, two  floors of type $s$
have always $a_{n+1}$ floor of type $l$ in between them and since, if $x_0 \in \Delta^{(n+1)}$, all points $x_i$ with $0\leq i \leq q_{n}-1$, belong to different floors of type $s$, we also have the following.
\begin{rem}\label{minimumdistancesmall}
If $x_0\in \Delta^{(n+1)}$, $x_j \geq j a_{n+1}\lambda^{(n)} \geq j \frac{a_{n+1}}{a_{n+1}+1}\lambda^{(n-1)}\geq\frac{j}{2}\lambda^{(n-1)} $ and  $y_j \geq  \frac{j}{2}\lambda^{(n-1)}$, for  $1\leq j \leq q_{n}-1$.
\end{rem}

Applying Corollary \ref{translatesxjyj} and Remark \ref{minimumdistance}, we can
control cancellations through a converging series and prove the following Lemma,
which shows that the main contribution to the sum along a cycle is determined by
the closest visits to $0$ and $1$.
\begin{lemma}\label{truncation}
For each $\epsilon>0$, there exists $k(\epsilon)$ such that for all $k\geq
k(\epsilon)$, if $ \BS{\alpha}{{x, f_1}}{q}$ is a sum along a
cycle of order $n$,
\bes
\left| \frac{1}{q}  \BS{\alpha}{{x, f_1}}{q} -  \frac{1}{q}
\sum_{i=0}^{k}
 \left ( \frac{1}{x_i} - \frac{1}{y_i} \right) \right|  \leq
 \epsilon.
\ees
\end{lemma}
\begin{proof}
Using the new labeling of the orbit points, introduced at the beginning of
\S\ref{cancellations}, we have 
\bes \BS{\alpha}{{x}}{q} = \sum_{i=0}^{q-1} \left(  \frac{1}{R_{\alpha}^i {x}} -
\frac{1}{1-R_{\alpha}^i {x}} \right) = \sum_{i=0}^{q-1} \left(  \frac{1}{x_i} -
\frac{1}{y_i}  
 \right).
\ees

Let us apply Corollary \ref{translatesxjyj} to control $\left|y_{i}- x_{i+1}
\right| $ or $\left| x_i -y_{i+1} \right| $ respectively.
 In the case ${x} \in \Delta^{(n)}$, rearranging the terms of the summation to use (\ref{differencen}), we get
\begin{eqnarray}
\label{cycleminusfirstpoints}
 \BS{\alpha}{{x}}{q} - \sum_{i=0}^k \left( \frac{1}{x_{i}}- \frac{1}{y_i} \right) 
 &=& \sum_{i=k}^{q-2} \left( \frac{1}{x_{i+1}} - \frac{1}{y_i} \right) + \frac{1}{y_k} - \frac{1}{y_{q-1}} \\ &=&\label{secondrearrang}
\sum_{i=k}^{q-2} \frac{{\lambda^{(n+1)}} -{2x_0}  }{y_i x_{i+1}} + \frac{1}{y_k} - \frac{1}{y_{q-1}} .
 \end{eqnarray}
From (\ref{secondrearrang}), using 
 Remark \ref{minimumdistance}
 which gives $x_i, y_i \geq i \lambda^{(n)}$, we have, as long as $k\geq 1$, 
\bes
\left| \frac{1}{q} \BS{\alpha}{{x}}{q}  - \frac{1}{q}\sum_{i=0}^{k} \left(
\frac{1}{x_i} - \frac{1}{y_i} \right) \right|
\leq
\left|\frac{1}{q
 \lambda^{(n)} }  \sum_{i=k}^{q-2}  \frac{ \frac{\lambda^{(n+1)}}{ \lambda^{(n)}
} -\frac{2x_0} {\lambda^{(n)} }}{i^2}\right| + \frac{1}{q y_k}.
 \ees
The second term in the RHS, by  Remark \ref{minimumdistance} and (\ref{towersareas}), is bounded by  $\frac{1}{q y_k} \leq$ ${(k q_{n+1} \lambda^{(n)})}^{-1} \leq$ $\frac{ 2}{k} $ and hence by ${\epsilon}/{2}$ if $k\geq k(\epsilon)$ where $k(\epsilon ) $ is large enough.  
Moreover, since $(q\lambda^{(n)})^{-1} = (q_{n+1} \lambda^{(n)})^{-1}\leq 2$ by (\ref{towersareas}), 
$\lambda^{(n+1)}/ \lambda^{(n)}\leq 1$ and $x_0/\lambda^{(n)}\leq 1$,
 the first term in the RHS is bounded by the remainder of a converging series. Hence, we can
choose $k(\epsilon)$ large enough so that also  $ 6 \sum_{k(\epsilon)}^{\infty} i^{-2} <
\epsilon/2$ and this concludes the proof in the case ${x}\in \Delta^{(n)}$.
In the case  ${x}\in \Delta^{(n+1)}$ and $q=q_n$, in an  analogous way we get
\bes
 \BS{\alpha}{{x}}{q} = \sum_{i=0}^k \left( \frac{1}{x_{i}}- \frac{1}{y_i} \right) 
+ \sum_{i=k+1}^{q-1} \left( \frac{1}{x_{i-1}} - \frac{1}{y_i} \right) - \frac{1}{x_{k}} + \frac{1}{x_{q-1}}
\ees
and this time by   (\ref{differencen+1}) and Remark \ref{minimumdistancesmall} and also (\ref{towersareas}) and $y_0\leq \lambda^{(n+1)} $, we have
\bes
\left|  \frac{1}{q}\sum_{i=k+1}^{q-1} \left( \frac{1}{x_{i-1}} - \frac{1}{y_i}
\right) \right| \leq
\left|\frac{4}{q}
  \sum_{i=k+1}^{q-1}  \frac{2  y_0 -  \lambda^{(n)} }{(\lambda^{(n-1)})^2
i^2}\right|\leq \frac{12}{q_n \lambda^{(n-1)}
} \sum_{i=k+1}^{q-1}  \frac{ \frac{\lambda^{(n)}}{\lambda^{(n-1)} } }{ i^2}\leq
\sum_{i=k+1}^{q-1}  \frac{24}{ i^{2}}.
 \ees
Moreover, using again Remark \ref{minimumdistancesmall},  one has $\frac{1}{q}\left|\frac{1}{x_{q-1}}- \frac{1}{x_{k}} \right| \leq \frac{1}{q x_k} \leq\frac{2}{ k\, q_{n}\lambda^{(n-1)}}\leq \frac{4}{k}$ so by choosing $k(\epsilon)$ large enough this concludes the proof also in this second case.
\end{proof}
\begin{cor}\label{cycleestimatelemma}  There exists $M>0$ such that for all sums
$\BS{\alpha}{{x}}{q}$ along a cycle
we get%
\be\label{estimatecycle}
\left| \frac{1}{q} \BS{\alpha}{{x}}{q}\right| \leq  \max \left\{\frac{1}{q
x_0}, \frac{1}{q y_0} \right\} + M.
\ee
\end{cor}
\begin{proof}
It follows from the estimates in the proof of Lemma \ref{truncation} for $k=0$,
if we take $M=24 \sum_{i=1}^{\infty} {i^{-2}} $.
\end{proof}

We are now ready to prove Proposition \ref{cycleprop}.

\begin{proofof}{Proposition}{cycleprop}
Given $\epsilon>0$, choose $k(\epsilon)$ so that Lemma \ref{truncation} holds. The value of $\BS{\alpha}{{x}}{q}/q$ is hence determined up to $\epsilon$ by the contribution from any number  $k\geq k(\epsilon)$ of closest points to $0$ and $1$. 
Let us show that the positions of $x_i$ and $y_i$ with $0\leq i\leq k(\epsilon)$ are determined by ${x}$ and $a_{n +2-  k }$ with $0\leq k\leq K(\epsilon)$ for some $K(\epsilon)$.  Choose $K= K(\epsilon)$ so that $2^{(K(\epsilon)-2)/2}> k(\epsilon)$. Hence, since the elements of $\xi^{(n)}$ of a fixed type contained inside $\Delta^{(n-K)}$ are at least\footnote{This estimate can be obtained from the recursive relations between $\xi^{(n)}$ and $\xi^{(n+1)}$, see \S\ref{partitionssec}, and the lower bound for the growth of Fibonacci numbers.} $2^{\frac{K-1}{2}}$, we have that
\begin{eqnarray}
&& \{ x_0,x_1,  \dots, x_{k(\epsilon)}\} \subset \Delta^{(n-K(\epsilon))}
, \nonumber \\ && \nonumber \{ y_0,y_1,  \dots, y_{k(\epsilon)}\} \subset \Delta^{(n-K(\epsilon)+1)} = [1- \lambda^{(n-K(\epsilon)+1)}).
\end{eqnarray}
Let $q g_n^{\epsilon} $ be the sum over all the points  $y_i$, $i=0, \dots, k-1$, which are contained in $\Delta^{(n-K(\epsilon)+1)}$ (here $k\geq k(\epsilon)$ denotes their cardinality) and over the corresponding points $x_i$, $i=0, \dots, k$, which are all contained in $\Delta^{(n-K(\epsilon))}$.  Explicitly, if ${x} \in \Delta^{(n)}$ (and hence $x=x_0$ and $q=q_{n+1})$
\be
\label{gdef}
 g_n^{\epsilon} ({x}, \alpha, q ) = \frac{1}{q} \left(  \frac{1}{x_0} + \sum_{i=0}^{k-1} \frac{{\lambda^{(n+1)}} -{2x_0}  }{y_i x_{i+1}}\right)
 = \frac{1}{q_{n+1} \lambda^{(n)}} \left(  \frac{1}{\frac{x_0}{\lambda^{(n)}}} + \sum_{i=0}^{k-1} \frac{
\frac{\lambda^{(n+1)}}{\lambda^{(n)}}  - \frac{2x_0}{\lambda^{(n)}}}{\frac{y_i }{\lambda^{(n)}} \frac{ x_{i+1}}{\lambda^{(n)}}}   \right) .
\ee
and a similar expression can be written for  ${x} \in \Delta^{(n+1)}$.

We remark that the  points $x_i$ with $0\leq i\leq k$ belong to floors of the tower $Z^{(n)}_l$ of the partition $\xi^{(n)}$ which are contained in $\Delta^{(n-K(\epsilon))}$ and are determined once $x_0$ and the relative position of these floors is given. Similarly, $y_i$ with $0\leq i\leq k$ belong to the floors of the tower $Z^{(n+1)}_0$ of the partition $\xi^{(n)}$ which are contained in $\Delta^{(n-K(\epsilon))}$
 and are determined once $y_0$ and the relative position of the floors are given.

By Remark \ref{firstlastpoints}, $x_0$ and $y_0$ can be expressed through $x, \lambda^{(n)}$, $\lambda^{(n+1)}$ and $a_{n+1}$.
Moreover, by Remark \ref{recursivepartitions}, the lengths $\lambda^{(n)}$ and  $\lambda^{(n+1)}$ and the entries $a_{n+2 - k }$ with $k\leq K(\epsilon)$ determine the  sequence of floors of  $\xi^{(n)}$ both inside $\Delta^{\left(n-K(\epsilon)\right)}$ and inside  $\Delta^{(n-K(\epsilon)+1)}$.
 Dividing all quantities by $\lambda^{(n)}$ if $x\in \Delta^{(n)}$ (or by $\lambda^{(n+1)}$ if $x\in \Delta^{(n+1)}$),
 the ratios  $x_i/\lambda^{(n)}$ and  $y_i/\lambda^{(n)}$ (or $x_i/\lambda^{(n+1)}$ and  $y_i/\lambda^{(n+1)}$), which are the quantities through which $g_n^{\epsilon} $ is expressed in (\ref{gdef}),  are determined by $x_0/\lambda^{(n)} = x/\lambda^{(n)}$ (or $y_0/\lambda^{(n+1)} = x/\lambda^{(n+1)}$),  $\lambda^{(n+1)}/\lambda^{(n)}$ and the entries $a_{n +2-  k }$ with $0\leq k\leq K(\epsilon)$.

Hence we have shown that $g_n^{\epsilon} $ can be expressed through the quantities in (\ref{gdependence}).
The relation (\ref{gapproximates}) follows immediately from Lemma \ref{truncation}.
\end{proofof}

\section{General Birkhoff sums.}\label{decompositionsec}

In this section we consider  general Birkhoff sums and  prove the following.
\begin{prop}\label{generalBSprop}
For each $\epsilon>0$, $\delta > 0 $, there exists a function $G^{\epsilon, \delta } = G^{\epsilon, \delta }(x,\alpha,  N)$, such
that for all sufficiently large $N$,
\be\label{Gapproximates}
Leb \left\{ (x ,\alpha ) \st \left| \frac{1}{N}  \BS{\alpha}{{x}}{N} -
G^{\epsilon, \delta }( {x},\alpha,  N) \right| \geq \epsilon \right\} \leq
\delta
\ee
and 
there exists $K_1=K_1(\epsilon )$ such that any $G^{\epsilon, \delta }({x},\alpha,  N )$  can be expressed as a
function  of the following quantities:
\be\label{Gdependence}\begin{split}
 G^{\epsilon, \delta} \left(\frac{q_{n}}{N},
\frac{{d(x)}}{\lambda^{(n-2)}}, \frac{h(x)}{q_{n-1}}, \frac{\lambda^{(n-
1)}}{\lambda^{(n-2)} },\frac{q_{n-2}}{q_{n-1}}, \frac{1}{q_{n-1} \lambda^{(n-2)} },  a_n, 
\dots, a_{n\pm  K_1} \right), \, \, \mathrm{if}\, {x}\in  Z^{(n-2)}_l; \\ G^{\epsilon, \delta}\left(\frac{q_{n}}{N},
\frac{{d(x)}}{\lambda^{(n-1)}},  \frac{h(x)}{q_{n-2}}, \frac{\lambda^{(n-
1)}}{\lambda^{(n-2)} },\frac{q_{n-2}}{q_{n-1}},  \frac{1}{q_{n-2} \lambda^{(n- 1)} }, a_n,
\dots, a_{n\pm K_1} \right), 
\, \, \mathrm{if} \, {x}\in  Z^{(n-2)}_s;
\end{split}
\ee
where $n:=n(\alpha,N)$ is as in (\ref{n(N)def}) and where $d(x)=d_{n-2}(x)$ and $h(x)=h_{n-2}(x)$ are defined as follows:\footnote{See \S\ref{relpos} 
 below for a geometric explanation of the meaning of $d(x)$
and $h(x)$.}
\begin{itemize}
\item[-]
 if $x\in \Delta^{(n-2)}_j \subset Z^{(n-2)}_l$, $d(x) \in \Delta^{(n-2)}_0 $ is
such that $R_{\alpha}^j d(x) = x$ and $h(x)=q_{n-1}-j$, for some $0\leq j < q_{n-1}$;
\item[-]
 if $x\in \Delta^{(n-1)}_j \subset Z^{(n-2)}_s$, then $d(x) \in \Delta^{(n-1)}_0
$ is such that $R_{\alpha}^j (1-d(x)) = x$ and $h(x)=q_{n-2}-j$, for some $0\leq j < q_{n-2}$.
\end{itemize}
\end{prop}
As a corollary of this Proposition, we prove that $G^{\epsilon, \delta }({x},\alpha,  N )$ has a joint {limiting} distribution in $(x,\alpha)$ as $N$ tends to infinity.
Indeed, all the quantities through which $G^{\epsilon, \delta}$ is expressed in (\ref{Gdependence}) have
limiting distributions as $N$ tends to infinity (some of them were considered in
\S\ref{renewalsec}, for the other ones see \S\ref{limitingsec}, Lemma
\ref{limitingratios}) and, together with the continuity properties of $G^{\epsilon, \delta }({x},\alpha,  N )$ (see Lemma \ref{discontinuities}),  this implies that
$G^{\epsilon, \delta}$ has a limiting distribution. 

In order to prove Proposition \ref{generalBSprop} (in \S \ref{proofofgeneralBSprop}), we
first show how to decompose $\BS{\alpha}{{x}}{N}$ into Birkhoff sums along
cycles, see \S\ref{cycledecomp}. Then, in \S\ref{finitelymany} we show that,
neglecting a  set of $(\alpha, {x})$ of small measure, we can reduce the
decomposition to finitely many cycles. Hence the function $G^{\epsilon, \delta}$ is
defined combining finitely many functions  $g_n^{\epsilon} $ constructed in Proposition
\ref{cycleprop} to approximate sums along cycles.

\subsection{Decomposition into cycles.}\label{cycledecomp}
Fix $N$ and $\alpha$ and let $n = n(N, \alpha)$ 
 be the unique \emph{even} $n\in \mathbb{N}$ such that $q_{n-2} < N < q_{n}$. The dependence on $N$ and $\alpha$ will be omitted in this section since $N$ and $\alpha$ are fixed throughout. Consider the partition $\xi^{(n-2)}$. We will decompose the orbit $\{R_{\alpha}^i x, \, 0\leq i < N \}$ into cycles and $n-2$ will be the biggest order of the cycles involved in the decomposition.

\subsubsection{Relative positions inside the towers.}\label{relpos}
The definition of $G^{\epsilon,\delta}$ depends on whether $x \in Z^{(n-2)}_l$ or  $x \in Z^{(n-2)}_s$. Throughout this section, the quantities $d(x)=d_{n-2}(x)$ and $h(x)=h_{n-2}(x)$ are defined as in Proposition \ref{generalBSprop} and locate the position of $x $ inside the tower of $\xi^{(n-2)}$ to which it belongs. If $x\in  Z^{(n-2)}_l$, there exists $0\leq j < q_{n-1}$ such that $x\in \Delta^{(n-2)}_j $; if $x_0 \in \Delta^{(n-2)}_0 $ is such that $R_{\alpha}^j x_0 = x$, then $d_{n-2}(x)=x_0$ is the distance from $0$ in the base floor. Similarly, if $x\in  Z^{(n-2)}_s$,  there exists $0\leq j < q_{n-2}$ such that $x\in \Delta^{(n-1)}_j $ and if $z_0 \in \Delta^{(n-1)}_0 $ is such that $R_{\alpha}^j z_0 = x$, then  $d_{n-2}(x)=1-z_0$ is the distance in the base from $1$. The quantity $h_{n-2}(x)$, which is given respectively by $q_{n-1}-j$ or  $q_{n-2}-j$, represents the distance of the floor to which $x$ belongs from the top of the tower. In particular, we remark that by construction $R_{\alpha}^{h_{n-2}(x)}x \in \Delta(n\!-\!2)$.

\subsubsection{Cycles of order $n-2$.}
Let us use the notation $\Orb{x}{N} := \{R_{\alpha}^i x, \, 0\leq i < N \}$ to denote orbit segments.
\begin{figure}
\centering
\includegraphics[width=0.95\textwidth]{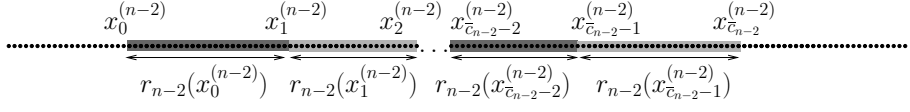}
\caption{Decomposition of the orbit $\Orb{x}{N}$ into cycles of order $n-2$.\label{orbitdecompBS}}
\end{figure}
Let us first locate inside  the orbit $\Orb{x}{N}$ all orbit segments which correspond to cycles along towers of $\xi^{(n-2)}$. As just remarked, $R_{\alpha}^{h_{n-2}(x)} x \in \Delta (n-2)= \Delta^{(n-2)}\cup \Delta^{(n-1)} $ and $h_{n-2}(x)$ is the first time $i$ for which $R_{\alpha}^i x \in \Delta(n-2)$. Let $x^{(n-2)}_0 := R_{\alpha}^{h_{n-2}(x)}x$. The following visits to $\Delta(n\!-\!2)$ can be expressed though $T^{(n-2)}$, first return map to  $\Delta(n\!-\!2)$ (see \ref{partitionssec}). Let 
\bes
x^{(n-2)}_i :=\left(T^{(n-2)}\right)^i  x^{(n-2)}_0; \qquad
\overline{c}_{n-2} : = \max \{ i \in \mathbb{N} \st \left(T^{(n-2)}\right)^i  x^{(n-2)}_0 \in \Orb{x}{N}\},
\ees
where $\overline{c}_{n-2}$ gives the number of visits of  $\Orb{x}{N}$ to  $\Delta(n\!-\!2)$.
Let $r_{n-2}(x)$ be the first return time of $x\in \Delta{(n-2)} $ to  $\Delta{(n-2)}$, i.e.~$r_{n-2}(x)= q_{n-1}$ if $x\in  \Delta^{(n-2)}$,  $r_{n-2}(x)= q_{n-2}$ if $x\in  \Delta^{(n-1)}$. Then each orbit segment
\bes
\Orb{x^{(n-2)}_i}{r_{n-2}(x^{(n-2)}_i)} =  \{R_{\alpha}^i x^{(n-2)}_i, \, 0\leq i < {{r}}_{n-2}(x^{(n-2)}_i ) \}
\ees
is a \emph{cycle of order $n-2$} and all cycles corresponding to $i=0,\dots, \overline{c}_{n-2} -1 $ are completely contained in $\Orb{x}{N}$, as in the representation of the orbit decomposition in Figure \ref{orbitdecompBS}.
Hence, so far
\be\label{orbitdecompn-2}
\begin{split}
\Orb{x}{N} & =  
\{R_{\alpha}^i x, \, 0\leq i < h_{n-2}(x) \} \,
\cup \,
\bigcup_{i=0}^{\overline{c}_{n-2}-1} \Orb{x^{(n-2)}_i}{{r}_n(x^{(n-2)}_i)} \\ & \cup  \,
\{R_{\alpha}^i (x ), \,h_{n-2}(x)+ \sum_{i=0}^{\overline{c}_{n-2}-1} {r}_{n-2}(x^{(n-2)}_i) \leq i < N \},
\end{split}
\ee
where the orbit segments which appear in the central union are cycles of order $n-2$. We will refer to the first and the last term as to the \emph{initial} and \emph{final} orbit \emph{segments} (see again Figure \ref{orbitdecompBS}).

Let us estimate the number $\overline{c}_{n-2}$ of cycles of order $n-2$. 
Since the cardinality of points in a cycle of order $n-2$ is at least $q_{n-2}$ and $q_n>N$ definition of $n$, $\overline{c}_{n-2} \leq N/q_{n-2}\leq  q_n/q_{n-2}$ and  from the recurrent relation $q_n = a_{n} q_{n-1} + q_{n-2}$ we have the following.
\begin{rem}\label{cyclesnbound}
$\overline{c}_{n-2} \leq ( a_n +1 )(  a_{n-1} +1 )$.
\end{rem}

\subsubsection{Further cycles.}
\begin{figure}
\centering
\includegraphics[width=0.95\textwidth]{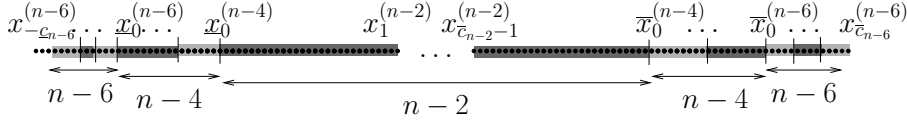}
\caption{Decomposition of the orbit $\Orb{x}{N}$ into cycles of further orders.}
\end{figure}
The initial and final segments of $\Orb{x}{N}$ in (\ref{orbitdecompn-2})
 will be decomposed by induction into cycles of even orders $m<n-2$.
Let   $ \overline{x}^{(n-4)}_{0}:=x^{(n-2)}_{\overline{c}_{n-2}} $ and  $\underline{x}^{(n-4)}_{0}:=x^{(n-2)}_{0} $ be respectively the first point of the last segment and the last point of the initial segment.
We remark that  we have  $\overline{x}^{(n-4)}_{0}, \underline{x}^{(n-4)}_{0} \in \Delta(n\!-\!4)$, since by construction of the partitions   $\Delta(n\!-\!2)\subset  \Delta(n\!-\!4)$.

Assume by induction that we have already subdivided the initial and final segments into cycles up to order $m+2$ ($m$ even) and let
$\overline{x}^{(m)}_{0}$ and  $\underline{x}^{(m)}_{0}$ be respectively the last point of the initial segment and the first point of the final segment, with $\underline{x}^{(m)}_{0}, \overline{x}^{(m)}_{0} \in  \Delta(m\!+\!2)\subset  \Delta(m)$. 
 Let
\begin{eqnarray}
{x}^{(m)}_i :={T^{(m)}}^i   \overline{x}^{(m)}_{0},\, \,  i=0,\dots,  \overline{c}_{m};  &&   \overline{c}_{m} : = \max \{ i \in \mathbb{N} \st {T^{(m)}}^i  \overline{x}^{(m)}_0 \in \Orb{x}{N}\}; \nonumber \\
{x}^{(m)}_{-i} :={T^{(m)}}^{-i}   \underline{x}^{(m)}_{0}, \, \,   i=1,\dots,  \underline{c}_{m};  &&   \underline{c}_{m} : = \max \{ i \in \mathbb{N} \st {T^{(m)}}^{-i}  \underline{x}^{(m)}_0 \in \Orb{x}{N}\}. \nonumber
\end{eqnarray}
The points  $ \{{x}^{(m)}_i, \, 0\leq i\leq {\overline{c}}_{m}\}$ give all the visits to $\Delta(m)$ which occur in the final orbit segment and  $\{ {x}^{(m)}_i, \, - \underline{c}_{m} \leq i \leq -1\}$ 
 give all the visits to $\Delta(m)$ which occur in the initial orbit segment. Moreover, by construction, 
\be\label{inclusionsvisits}
{x}^{(m)}_i \in \Delta(m)\backslash \Delta(m+2), \qquad   - \underline{c}_{m} \leq i \leq -1\ , \quad    1 \leq  i\leq {\overline{c}}_{m}.
\ee

Let $r_{m}(x)$ be as before the first return time of $x\in \Delta(m)$ to $\Delta(m)$.
Thus, all orbit segments
\bes \Orb{{x}^{(m)}_i}{r_{m}({x}^{(m)}_i)}, \quad  i=0,\dots, \overline{c}_{m}-1,
\quad i=-\underline{c}_{m}, \dots, -1,
\ees
are cycles of order $m$ which are completely contained in $\Orb{x}{N}$.  Moreover, since by (\ref{inclusionsvisits}) the initial and final segment do not contain any visit to $\Delta(m\!+\!2)$, except 
 $\overline{x}^{(m)}_{0}$ and since the points in $[0,1)\backslash \Delta(m\!+\!2)$ have at least distance $\lambda^{(m+3)}$ from $0$ and $1$, we also have the following.
\begin{rem}\label{minimumdistance01}
The distance from $0$ and from $1$ of any of the points in the orbit segments $\Orb{{x}^{(m)}_i}{r_{m}({x}^{(m)}_i)}$ for  $-\underline{c}_{m}\leq i \leq  -1$ and $1\leq i \leq \overline{c}_m-1$ is at least $\lambda^{(m+3)}$.
\end{rem}
In the previous Remark, the orbit corresponding 
 to $i=0$ was excluded since it contains $\overline{x}^{(m)}_{0}\in \Delta(m+2)$.

Let $\overline{x}^{(m-2)}_{0}:={x}^{(m)}_{\overline{c}_{m}} $ and  $\underline{x}^{(m-2)}_{0}:={x}^{(m)}_{-\underline{c}_{m}}$ and continue by induction to decompose the remaining initial and final segments.  If, for some $m$, we have $\underline{c}_{m}=\overline{c}_{m}=0$, there are no cycles of order $m$ in the decomposition. If either $\underline{c}_{m}$ or $\overline{c}_{m}$ are not zero, we say on the contrary that the order $m$ \emph{is present} in the decomposition. \label{present}


We get the following decomposition of the whole orbit into cycles:
\bes
\Orb{x}{N}= \bigcup_{\begin{subarray}{c}{{m=0}}\\{m \, \mathrm{even}}\end{subarray}}^{n-2} \bigcup_{i=-\underline{c}_m}^{\overline{c}_m-1}  \Orb{{x}^{(m)}_i}{r_{m}({x}^{(m)}_i)},
\ees
where for uniformity of notation, we set $\underline{c}_{n-2}:=0$ and the union for a given $m$ has to be considered empty when $\overline{c}_m=\underline{c}_m=0 $.  
Since by construction the length of the initial and final segment after the decomposition of order $m+2$  is at most $q_{m+3}$
and each cycle of order $m$ has length at least $q_{m}$, using that $q_{m+1}/q_{m}\leq a_{m+1}+1$, we have the following.
\begin{rem} \label{cardinalitycycles}
The number $\overline{c}_m + \underline{c}_m$ of cycles of order $m$ in the decomposition satisfies
\bes
\overline{c}_m + \underline{c}_m\leq2 (a_{m+3} +1)(a_{m+2}+1)(a_{m+1} +1).
\ees
\end{rem}

\subsection{Reduction to finitely many cycles.}\label{finitelymany}
The decomposition in \S\ref{cycledecomp} implies that:
\be\label{BSdecomp}
 \BS{\alpha}{{x}}{N} =  \sum_{\begin{subarray}{c}{{m=0}}\\{m \, \mathrm{even}}\end{subarray}}^{n(N,\alpha)-2} \sum_{i=-\underline{c}_m}^{\overline{c}_m-1}  \BS{\alpha}{{x}^{(m)}_i}{r_{m}({x}^{(m)}_i)}.
\ee
In order to construct a good approximation of $\frac{1}{N}\BS{x_0}{\alpha}{N}$ in measure, it is enough to consider in the previous expression only a fixed and finite number of cycles.
\begin{prop}\label{finitecycleslemma}
For each $\epsilon>0$ and $\delta>0$  there exist an even integer $M\geq 2$ and $N_1\in \mathbb{N}$ such that, for all $N\geq N_1$, we have
\be\label{finitecycles}
Leb \left\{ (x ,\alpha ) \st \left| \frac{1}{N}  \BS{\alpha}{{x}}{N} -   \frac{1}{N} \sum_{\begin{subarray}{c}{{m=n(N,\alpha)-M}}\\{m \, \mathrm{even}}\end{subarray}}^{n(N,\alpha)-2} \sum_{i=-\underline{c}_m}^{\overline{c}_m-1}  \BS{\alpha}{{x}^{(m)}_i}{r_{m}({x}^{(m)}_i)}
 \right| \geq \epsilon \right\} \leq \delta,
\ee
where we adopt the convention that if $n(\alpha,N)<M$ the sum in  (\ref{finitecycles}) runs from $m=0$. 
\end{prop}
The reason why the Proposition holds is that the contributions from different cycles decay exponentially in the order.  The set of small measure which needs to be neglected contains the set of $\alpha$ for which there are too many cycles of some orders and the set of initial points for which the contribution of the cycles of order $n-2$ is too large.

In the proof of Proposition \ref{finitecycleslemma} we will use the following.
\begin{lemma}\label{powerslemma}
For each $\delta>0$ there exist a constant $C=C(\delta)$ and $N_0 = N_0 (\delta)> 0 $  so that for all $N\geq N_0$ there exists a set $A_N=A_N(\delta)\subset [0,1)$ such that $Leb(A_N)< \delta $ and, for all $\alpha \in [0,1)\backslash A_N$, 
\be\label{apowers}
a_{n(N,\alpha)-k} \leq C (k+1)^2 , \qquad  0\leq k < n(N,\alpha).
\ee
\end{lemma}
The proof of Lemma \ref{powerslemma} relies on the techniques used in \cite{SU:ren}. We postpone the proof to the Appendix, \S\ref{proofpowerslemma}.

\begin{proofof}{Proposition}{finitecycleslemma}
Let us consider the difference which we want to estimate to get (\ref{finitecycles}). Let us denote, for brevity, $n=n(N,\alpha)$.
 From (\ref{BSdecomp}), applying Remark \ref{cardinalitycycles} to estimate the number of cycles of each order and keeping aside the term $i=0$, 
we get:
\be\label{allcycles} \begin{split}
& \left| \frac{1}{N}
 \sum_{\begin{subarray}{c}{{m=0}}\\{m \, \mathrm{even}}\end{subarray}}^{n-M-2} \sum_{i=-\underline{c}_m}^{\overline{c}_m-1}  \BS{\alpha}{{x}^{(m)}_i}{r_{m}({x}^{(m)}_i)} \right|   \leq \frac{1}{N}
 \sum_{\begin{subarray}{c}{{m=0}}\\{m \, \mathrm{even}}\end{subarray}}^{n-M-2} 
\left| { \BS{\alpha}{{x}^{(m)}_0}{r_{m}({x}^{(m)}_0)}}  \right| 
\\   &\phantom{=====} +
  \sum_{\begin{subarray}{c}{{m=0}}\\{m \, \mathrm{even}}\end{subarray}}^{n-M-2}
 \frac{{r_{m}({x}^{(m)}_i)}}{N}2\prod_{s=1}^3(a_{m+s} +1)
 \frac
{ \ \left|\BS{\alpha}{{x}^{(m)}_i}{r_{m}({x}^{(m)}_i)} \right|}{{r_{m}({x}^{(m)}_i)}}
  .
\end{split}\ee
 To estimate the contribution given by each cycle, let us apply  Corollary \ref{cycleestimatelemma}. When  $i\neq 0$, Remark \ref{minimumdistance01} gives a lower bound on the contribution coming from closest points and we get
\be\label{cyclecontribution}
\frac{ \left|\BS{\alpha}{{x}^{(m)}_i}{r_{m}({x}^{(m)}_i)}\right|}{{r_{m}({x}^{(m)}_i)}}
 \leq  \frac{1}{r_{m}({x}^{(m)}_i) \lambda^{(m+3)}} + M \leq 2 \prod_{s=0}^3(a_{m+s}+1) +M,\quad i\neq 0,
\ee
where in the last inequality we used that
\be\label{areaproduct}
\frac{1}{r_{m}({x}^{(m)}_i) \lambda^{(m+3)}}
\leq\frac{1}{q_m \lambda^{(m-1)}}
 \prod_{s=0}^3 \frac{\lambda_{m+s-1}}{\lambda^{m+s}}\leq 2 \prod_{s=0}^3(a_{m+s}+1) .
\ee


Since $ q_{n-2}\leq N$, ${r_{m}({x}^{(m)}_i)}\leq q_{m+1}$ 
 and $q_{n+2s}\geq {2^s}q_n$ (from the recurrent relations (\ref{recursiverelations})), assuming $M\geq 2$, we have
\be\label{coefficients}
 \frac{{r_{m}({x}^{(m)}_i)}}{N}\leq \frac{1}{2^{n-m-4}},\qquad  0\leq m\leq n-4.
\ee

When $i=0$, for any order $m$ which is present in the decomposition (i.e.~$\underline{c}_m>0$), 
 let us estimate the contribution of the closest points 
 to $0$ and $1$ in $\Orb{{x}^{(m)}_i}{r_{m}({x}^{(m)}_i)}$ with the distance from  $0$ and $1$ of elements of the \emph{last} (i.e.~the minimum $m'>m$) order $m'$ which is present. 
To get an upper estimate, let us consider the worst case in which all orders $m <  n-M$ are present. For $m< n-M-2$, a lower bound for minimum distance from $0$ or $1$  of points of order $m+2$ (which belong to $[0,1)\backslash\Delta(m+4)$ by (\ref{inclusionsvisits})) is given by $ \lambda^{(m+5)}$. For $m=n-M- 2$, let us denote the  minimum between the distance of $x_0^{(n-M-2)}$ from $0$ and $1$ by  $m_{n-M-2}(x)$. 
Thus, applying Corollary \ref{cycleestimatelemma} we get
\be\label{cyclecontributionspeciali}
\begin{split} &
\left|\sum_{\begin{subarray}{c}{{m=0}}\\{m \, \mathrm{even}}\end{subarray}}^{n-M-2} 
  \frac{ \BS{\alpha}{{x}^{(m)}_0}{r_{m}({x}^{(m)}_0)}}{N}
\right| \leq
\\ &  
 \frac{1}{N m_{n-M+2}(x)}+
\sum_{\begin{subarray}{c}{{m=0}}\\{m \, \mathrm{even}}\end{subarray}}^{n-M-4}  \frac{{r_{m}({x}^{(m)}_i)}}{N}  \frac{1}{{r_{m}({x}^{(m)}_i) \lambda^{(m+5)}}}   + 
 \sum_{\begin{subarray}{c}{{m=0}}\\{m \, \mathrm{even}}\end{subarray}}^{n-M-2}   \frac{{r_{m}({x}^{(m)}_i)}}{N} 
M 
\leq \\ &
 \frac{1}{N m_{n-M+2} (x)}
+\sum_{\begin{subarray}{c}{{m=0}}\\{m \, \mathrm{even}}\end{subarray}}^{n-M-4} \frac{2 \prod_{s=0}^5(a_{m+s}+1)}{2^{n-m-4}}
+ 
 \sum_{\begin{subarray}{c}{{m=0}}\\{m \, \mathrm{even}}\end{subarray}}^{n-M-2}   \frac{
M }{2^{n-m-4} } 
\end{split}
\ee
where in the last inequality we used (\ref{coefficients}) and an estimate analogous to (\ref{areaproduct}).

Combining
(\ref{cyclecontribution}), (\ref{coefficients}) and (\ref{cyclecontributionspeciali}) we have the following upper estimate for (\ref{allcycles}):
\be\label{finalestimatereminder}
\begin{split}
& \sum_{\begin{subarray}{c}{{m=0}}\\{m \, \mathrm{even}}\end{subarray}}^{n-M-2}
\frac{ 2\prod_{s=1}^3(a_{m+s}+1) ( M +  2\prod_{s=0}^3(a_{m+s}+1)) + M   }{2^{n-m-4}} + \\
& \qquad + \sum_{\begin{subarray}{c}{{m=0}}\\{m \, \mathrm{even}}\end{subarray}}^{n-M-4}
\frac{  2\prod_{s=0}^5(a_{m+s}+1) }{2^{n-m-4}} + \frac{1}{ N {m_{n-M+2}(x)}}.
\end{split}
\ee

Let us now prove (\ref{finitecycles}). Fix $\epsilon>0, \delta> 0$.
By Lemma  \ref{powerslemma}, for some $N_0$, we can choose for each $N\geq N_0$, a set $A_N\subset [0,1)$ such that $Leb(A_N)< \delta/2 $ and  if $\alpha\notin A_N$,  $a_{n-k} \leq C (k+1)^2$,  for $n=n(N,\alpha)$. Let us estimate (\ref{finalestimatereminder}). Since all terms are positive, we get an upper estimate by making both series run from $0$ to $n\!-M\!-\!2$. 
Hence, the first term in (\ref{finalestimatereminder}),  defining  for brevity 
$L(x_0,x_1, \dots,x_5)=  2 \prod_{s=1}^3(x_s+1) (M+ 2 \prod_{s=0}^3(x_s+1) ) + M + 2\prod_{s=0}^5(x_s+1) $, 
can be written 
 and estimated as follows:
\bes\begin{split}
 \sum_{\begin{subarray}{c}{{m=0}}\\{m \, \mathrm{even}}\end{subarray}}^{n-M-2} \frac{ L(a_{m}, \dots, a_{m+5})}{2^{n-m-4}}& \leq
 \sum_{\begin{subarray}{c}{{m=0}}\\{m \, \mathrm{even}}\end{subarray}}^{n-M-2} \frac{ L(C (n\!-\!m\!+\!1)^2,\dots , C (n\!-\!m\!-\!4)^2)}{2^{n-m-4}}  \leq \\& \leq
 \sum_{\begin{subarray}{c}{{m=0}}\\{m \, \mathrm{even}}\end{subarray}}^{n-M-2}
\frac{ |{P}(n-m)|}{2^{n-m-4}}
\leq  \sum_{\begin{subarray}{c}{{k=M+2}}\\{k \, \mathrm{even}}\end{subarray}}^{\infty}\frac{ |{P}(k)|}{2^{k-4}}
\end{split}
\ees
for some polynomial $P(x)$. Hence, choosing $M$ large enough, we can assure that the remainder of the series is less than $\epsilon/2$.

In order to conclude the proof, we still need to estimate the second term in (\ref{finalestimatereminder}). Let us first estimate the expectation of the quantity $(N {m_{n-M+2}(x)})^{-1}$. Let $E$ denote the conditional expectation with respect to the Lebesgue measure on the $x$ variable, for $\alpha$ and $N$ (and hence $n=n(N,\alpha)$) fixed. Let us show that 
 there exists a set $X_n\subset[0,1)$, with $Leb(X_n)\leq \delta/4$, such that, if $\chi_{X_n^C}$ denotes the characteristic function of $X^C_n$,
\be\label{expectation}
E \left( \frac{ \chi_{X_n^C} (x) }{N {m_{n-M+2}(x)}} 
 \right) \xrightarrow{M\rightarrow \infty} 0.
\ee
Let us recall that $m_{n-M-2}(x) = \min \{ x_0^{(n-M-2)}, 1- x_0^{(n-M-2)}\}$ and that $x_0^{(n-M-2)} = x_{\overline{c}_{n-M}}^{(n-M )}$. Since by construction  $x_{\overline{c}_{n-M}}^{(n-M )}$ is the last visit to $\Delta{(n-M )}$  in $\Orb{N}{x}$ and such visits are not more than $q_{n-M+1}$ apart, we have $x_0^{(n-M-2)} \in \{ R_{\alpha}^{N-i} x \st 1\leq i \leq q_{n-M+1} \}$ and hence 
\be\label{inclusion} 
m_{n-M-2}(x) \geq \min \{R_{\alpha} ^{N-i} x ,  1-R_{\alpha} ^{N-i} x , 1\leq i \leq q_{n-M+1}\}.
\ee

Consider the set 
\bes
X_n = \bigcup_{i=1}^{q_{n-M+1}} R_{\alpha}^{-N+i} \left( \left[ 0,\frac{\delta}{8q_{n-M+1}} \right] \cup \left[ 1-\frac{\delta}{8 q_{n-M+1}}, 1\right]\right). 
\ees
Clearly $Leb(X_n)\leq \delta/4$ and if $x\notin X_n$, by (\ref{inclusion}) we have $m_{n-M-2}(x) \geq \delta/8q_{n-M+1}$.
Hence, since $q_{n-2}<N$,
\bes
E \left( \frac{ \chi_{X_n^C} (x) }{N {m_{n-M+2}(x)}} \right) \leq E \left( \frac{ \chi_{X_n^C} (x) }{q_{n-2} {m_{n-M+2}(x)}}\right) \leq \frac{8}{\delta}\frac{q_{n-M+1}}{q_{n-2}},
\ees
from which (\ref{expectation}) follows using that  $q_{n+2s}\geq {2^s}q_n$ by  the recursive relations (\ref{recursiverelations}).

Hence, enlarging again $M$ if necessary, we can hence assure that the expectation in (\ref{expectation}) is less than $\epsilon \delta/8$. 
Remark moreover that the choice of $M$ depends on $\epsilon $ and $\delta $ only and is uniform in $n$ and $N$.
 
Let us denote by  $X_n'=\{  x \st \frac{1}{N m_{n-M+2}(x)} \geq \epsilon/2\}$. Thus,
using Chebyshev inequality, we get 
 that
\bes
Leb (X_n') \leq 
\frac{ E\left(\frac{ \chi_{X^C_n}(x) }{N m_{n-M+2} (x)}\right)}{\epsilon/2}  + Leb(X_n)  \leq  \frac{\delta}{2}.
\ees
This shows that also the second term of (\ref{finalestimatereminder}) is $\epsilon/2$-small when $\alpha \notin A_N$ and $x\notin X'_{n(\alpha, N)}$ and hence concludes the proof of the Lemma.
\end{proofof}

\begin{proofof}{Proposition}{generalBSprop}\label{proofofgeneralBSprop}
Let us define the function $G^{\epsilon, \delta}$ using the truncated decomposition into cycles. Given $\epsilon$ and $\delta$, let $M$ be given by Proposition \ref{finitecycleslemma} applied to $\epsilon/2$ and $\delta/2$ and let  $g_n^{\varepsilon} $ and $K$ be as in Proposition \ref{cycleprop}, relative to some $\varepsilon$ which will be fixed below as a function of $\epsilon$. Then define
\be\label{defG}
G^{\epsilon, \delta} (x, \alpha, N):=
\sum_{\begin{subarray}{c}{{m=n(\alpha,N)-M}}\\{m \, \mathrm{even}}\end{subarray}}^{n(\alpha,N)-2} \sum_{i=-\underline{c}_m}^{\overline{c}_m-1} \frac{r_m({x}^{(m)}_i)}{N}
 g^{\varepsilon}_m\left( \alpha, {x}^{(m)}_i, r_{m}({x}^{(m)}_i)\right)
\ee
with the convention that the sum runs from $m=0$ if $n(\alpha,N)< M$. 
 Let $K_1= M+ K$.
The estimate on ${r_m({x}^{(m)}_i)}/{N}$ given by (\ref{coefficients}) holds for  $m\leq n-4$.
Hence,   by Proposition \ref{cycleprop} and  Remarks \ref{cyclesnbound} and \ref{cardinalitycycles}, on the complement, denoting $n=n(N,\alpha)$,   we have  
\be\label{mainerror}\begin{split}
\left| G^{\epsilon, \delta}(x, \alpha, N)
-  \sum_{\begin{subarray}{c}{{m=n-M}}\\{m \, \mathrm{even}}\end{subarray}}^{n-2} \sum_{i=-\underline{c}_m}^{\overline{c}_m-1} \frac{r_m({x}^{(m)}_i)}{N} \frac{ \BS{\alpha}{{x}^{(m)}_i}{r_{m}({x}^{(m)}_i)} }{r_m({x}^{(m)}_i) }\right| \leq \\\leq \varepsilon \left( {\frac{q_{n-1}}{N}
 (a_{n} +1)(a_{n-1} +1)} 
+ \sum_{\begin{subarray}{c}{{m=n-M}}\\{m \, \mathrm{even}}\end{subarray}}^{n-4} \frac{
2\prod_{s=0}^3(a_{m+s} +1)}{2^{n-m-4}}\right).
\end{split}
\ee
By Theorem \ref{main},  neglecting a subset of $\alpha$ of Lebesgue measure less than $\delta/2$, we can also assume that ${q_{n}}/{N}\leq C$ and $a_n, a_{n-1}\leq A$ for some $C>1, A> 1$, so that the first term in the upper bound of (\ref{mainerror}) is less than $C(A+1)^2$.  
If moreover $\alpha \notin A_N$ where $A_N=A_N(\delta/2)$ is the set given  by Lemma \ref{powerslemma}, since, reasoning again as in the proof of Proposition \ref{finitecycleslemma},  the terms at numerators in the RHS are bounded by a polynomial in $(n-m)$, the series in the RHS is converging. Thus, choosing $\varepsilon$ appropriately, (\ref{mainerror}) is less than $\epsilon/2$. Combining  (\ref{mainerror}) with Proposition \ref{finitecycleslemma}, we proved (\ref{Gapproximates}).

Let us show that $G^{\epsilon, \delta}$ can be expressed as a function of the variables in (\ref{Gdependence}). The variable ${q_n}/{N}$ already appears explicitly in front of the sum.
Let us show that, for each $m$ involved in the sum,  the quantities $\underline{c}_m$, ${\overline{c}_m}$ and ${x}^{(m)}_i$, $r_{m}({x}^{(m)}_i)$ for $-\underline{c}_m \leq i \leq {\overline{c}_m}$, 
as they appear in the sum, can be expressed in the desired form.

 Let us use induction on $m$.  When $m=n\!-\!2:=\! n(\alpha,N)\!-\!2$, if we denote $ x^{(n-2)}_{-1} := (T^{(n-2)})^{-1}  x^{(n-2)}_{0}$,  we have by construction that $x^{(n-2)}_{-1}= d(x)$ or  $1-d(x)$, according to whether $x\in Z_l^{(n-2)}$ or  $Z_s^{(n-2)}$. 
The other points $x^{(n-2)}_{i}$, $i\geq 0$, are completely determined by the orbit under  $T^{(n-2)}$. Hence,
the ratios $x^{(n-2)}_{i}/\lambda^{(n-2)}$ (which
appear as variables of the functions $g^{\epsilon}_{n-2}$,  see Proposition \ref{cycleprop})
are determined by $d(x)/\lambda^{(n-2)}$ (or  $d(x)/\lambda^{(n-1)}$, according to whether $x\in Z_l^{(n-2)}$ or  $Z_s^{(n-2)}$) and $\lambda^{(n-1)}/\lambda^{(n-2)}$. The other variables of $g^{\epsilon}_{n-2}$, i.e.~$(\lambda^{(n-2)}q_{n-1})^{-1}$ (or $(\lambda^{(n-1)}q_{n-2})^{-1}$)  and   $a_{n+2-s}$, $0\leq s \leq K$, are already given, since $ K\leq K_1$.
Similarly, the orbit under the map  $T^{(n-2)}$ determines also whether $x^{(n-2)}_{i}\in \Delta^{(n-1)}_0$ or  $x^{(n-2)}_{i}\in \Delta^{(n-2)}_0$ and therefore  $r_{n-2}({x}^{(n-2)}_i)$. The number of iterations  ${\overline{c}_{n-2}}$ can be expressed as the maximum  $c$ such that $h(x)+\sum_{i=1}^c r_{n-2}({x}^{(n-2)}_i) \leq N  $, hence it involves ratios of type $q_{n-2}/N$ and $q_{n-1}/N$ (or equivalently $q_{n}/N$ and $q_{n-2}/q_{n-1}$ and $a_{n}$) and $h(x)/N$. This concludes the base of the induction.

For the inductive step, from $m+2$ to $m$, we remark, as before, that   ${x}^{(m)}_i$ and   $r_{m}({x}^{(m)}_i)$ are completely determined by the initial points  ${\overline{x}}^{(m)}_0$ and  ${\underline{x}}^{(m)}_0$ and by the positive and negative orbit of the induced map $T^{(m)}$. Moreover, the ratio $\lambda^{(m+1)}/\lambda^{(m)}$ is determined by the previous ratio $\lambda^{(m+3)}/\lambda^{(m+2)}$ and the entries $a_{m+2}$, $a_{m+3}$, while   ${\overline{x}}^{(m)}_0/\lambda^{(m)}$ and  ${\underline{x}}^{(m)}_0/\lambda^{(m)}$ (or respectively ${\overline{x}}^{(m)}_0/\lambda^{(m+1)}$ and  ${\underline{x}}^{(m)}_0/\lambda^{(m+1)}$ ) can be obtained from   ${{x}}^{(m+2)}_{\overline{c}_{m+2}}/\lambda^{(m+2)}$ and  ${{x}}^{(m+2)}_{-\underline{c}_{m+2}}/\lambda^{(m+2)}$ (or  ${{x}}^{(m+2)}_{\overline{c}_{m+2}}/\lambda^{(m+3)}$ and  ${{x}}^{(m+2)}_{-\underline{c}_{m+2}}/\lambda^{(m+3)}$), given by the inductive step,  and  $\lambda^{(m+2)}/\lambda^{(m)}$. In particular, the function $g^{\varepsilon}_{m}$ is a function ratios 
 which can be expressed in terms of the above quantities. Reasoning as for $m=n-2$, also  $r_{m}({x}^{(m)}_i)$ are determined by the orbit of the induced map  $T^{(m)}$. Analyzing the decomposition in \S\ref{cycledecomp}, one can see that the numbers  $\underline{c}_m$ and $\overline{c}_m$ of cycles can be expressed respectively as the biggest integers  $\underline{c}$ and $\overline{c}$ such that
\begin{eqnarray}
&&\sum_{\begin{subarray}{c} m+2\leq k \leq n-4\\ k \, \mathrm{even}\end{subarray}} \sum_{i=-\underline{c}_k}^{-1} r_{k}({x}^{(k)}_{i}) + \sum^{-1}_{i=-\underline{c}} r_{m}({x}^{(m)}_{i}) \leq h(x)  ; \\
&&\sum_{\begin{subarray}{c} m+2\leq k \leq n-2\\ k \, \mathrm{even}\end{subarray}} \sum_{i=0}^{\overline{c}_k-1} r_{k}({x}^{(k)}_{i}) + \sum_{i=0}^{\overline{c}-1} r_{m}({x}^{(m)}_{i}) < N- h(x).
\end{eqnarray}
Hence, dividing by $N$, one sees that they are determined by  ratios which, by inductive assumption, are already expressed as desired and by ratios of type $q_{m}/N$ and $q_{m+1}/N$, which can be determined from them, involving also  $a_{m+2}$, $a_{m+3}$. The other variables which appear in  $g^{\varepsilon}_{m}$ by Proposition \ref{cycleprop} are $a_{m+2-s}$ for $0\leq s\leq K$, $n\!-\!M\leq m < n$, which are included as variables thanks to the definition of $K_1$. This concludes the induction.
\end{proofof}

\section{Existence of the limiting distribution.}\label{limitingsec}
\subsubsection*{Limiting distributions of relative positions in the towers.}\label{relativepositionsec}
Let us consider the variables from which $G^{\epsilon, \delta}$ depends (see (\ref{Gdependence})). 
Give $ \alpha, N$, let $n=n(N,\alpha)$ and let $d(x)=d_{n-2}(x)$ and $h(x)= h_{n-2}(x)$ be defined as in \S\ref{decompositionsec}.
Consider the random variables on $[0,1)\times[0,1)$
\begin{eqnarray} && D_N (x,\alpha ) = \frac{{d_{n(N,\alpha)-2}(x)}}{\lambda^{(n-2)}} \chi_{Z^{(n-2)}_l} (x,\alpha)+ \frac{{d_{n(N,\alpha)-2}(x)}}{\lambda^{(n-1)}} \chi_{Z^{(n-2)}_s} (x,\alpha);\label{DN} \\
&& H_N (x,\alpha ) = \frac{{h_{n(N,\alpha)-2}(x)}}{q_{n-1}} \chi_{Z^{(n-2)}_l}(x,\alpha) +   \frac{{h_{n(N,\alpha)-2}(x)}}{q_{n-2}} \chi_{Z^{(n-2)}_s}(x,\alpha);\label{HN}\\
&& T_N (x,\alpha ) = {(q_{n-1}\lambda^{(n-2)})}^{-1} \chi_{Z^{(n-2)}_l} (x,\alpha)+ {(q_{n-2}\lambda^{(n-1)})}^{-1} \chi_{Z^{(n-2)}_s} (x,\alpha);\label{TN}
\end{eqnarray}
where $\chi_{Z^{(n-2)}_\omega}$ ($\omega=l$ or $s$) denotes the indicator function of the towers, i.e.  
\be
\chi_{Z^{(n-2)}_\omega}(x,\alpha)=\left\{ \begin{array}{ll}  1 & \mathrm{iff} \, \, x\in {Z^{(n-2)}_\omega}, \quad n=n(N,\alpha); \\ 0 & \mathrm{otherwise}. \end{array}\right.
\ee
The quantities $D_N$ and  $H_N$ locate the relative position of $x $ inside the tower of $\xi^{(n-2)}$ to which it belongs, while $T_N$ give the total measure of the tower.
We also remark that  $D_N, H_N, T_N$ can be used as variables in the expression for $G^{\epsilon, \delta}$ given in (\ref{Gdependence}).
\begin{lemma}\label{limitingratios}
$ D_N (x,\alpha )$, $ H_N (x,\alpha )$ and $ T_N (x,\alpha )$ have  limiting distributions as $N$ tend to infinity.
\end{lemma}
\begin{proof}
Since  $x$ is uniformly distributed, on each event $\{ (x,\alpha) \in {Z^{(n-2)}_w} \}$, $w=l,s$, the random variables  ${d_{n(N,\alpha)}(x)}$ and  ${h_{n(N,\alpha)}(x)}$ are renormalized so that $ D_N (x,\alpha )$ and  $ H_N (x,\alpha )$ are uniformly distributed on $[0,1)$ for each $N$.  Since $Leb \{ x \in Z^{(n-2)}_l\} = \lambda^{(n-2)}q_{n-1}$ and $Leb \{x \in Z^{(n-2)}_s\} = \lambda^{(n-1)}q_{n-2}$,  the existence of the limiting distribution of  $T_N$ follows from Corollary \ref{limitingquantities}.
\end{proof}

\subsection{Tightness and final arguments.}
\begin{lemma}\label{tightness}
The random variables $G^{\epsilon,\delta}$ are uniformly tight in $\epsilon, \delta$ and $N$, i.e.
\bes
\inf_{\begin{subarray}{c}  \epsilon, \delta<1\\ N\in \mathbb{N}\end{subarray}} Leb \{ (x,\alpha)\st  G_N^{\epsilon,\delta} (\alpha ,x,  N) \leq T \} \xrightarrow{T \rightarrow +\infty} 1.
\ees
\end{lemma}
\begin{proof}
Let us remark that, from the definition of  $G^{\epsilon,\delta}$ in (\ref{defG}), as $\epsilon$ and $\delta$ tends to zero (and hence $M \rightarrow \infty$), $G^{\epsilon,\delta}$ has the same structure, but more and more terms are present in the series while at the same time each function $g_m^{\varepsilon}$ in the series involves more and more variables (i.e.~$K\rightarrow \infty$ in (\ref{gdependence})).

Since $\inf_{\epsilon,\delta, N}Leb\{ G_N^{\epsilon,\delta}
\leq  T \} \geq Leb\{ \sup_{\epsilon,\delta, N}G_N^{\epsilon,\delta}\leq T \} $, it is enough to estimate the latter. 
Let us estimate $G^{\epsilon,\delta}$  arguing as in the proof of Proposition \ref{finitecycleslemma} to prove (\ref{finalestimatereminder}), hence using (\ref{coefficients}),  Remark \ref{cardinalitycycles} and moreover (\ref{gapproximates}) in order to apply Corollary \ref{cycleestimatelemma}. For each $\delta>0$, we can find by Lemma \ref{powerslemma} a $C(\delta)>1$ and set  $A_N(\delta) \subset[0,1)$ on which (\ref{apowers}) holds for $N$ sufficiently large and, by  Theorem \ref{main},  a $C'=C'(\delta)>0$ and $A_N'(\delta)\subset[0,1)$ such that $q_{n}/N \leq C'$ and $a_n$, $a_{n-1}\leq C'$ on $[0,1)\backslash A_N'$.

Hence, for each $\alpha \notin A_N\cup A_N'$,
\begin{eqnarray}
\nonumber
&&\sup_{\begin{subarray}{c}\epsilon<1\\ \delta <1\end{subarray}} \left| G^{\epsilon, \delta}(x,\alpha, N) \right| \leq  
 \sum_{\begin{subarray}{c}{{m=0}}\\{m \, \mathrm{even}}\end{subarray}}^{n-2} \sum_{i=-\underline{c}_m}^{\overline{c}_m-1} \frac{r_m({x}^{(m)}_i)}{N}
 \sup_{\varepsilon <1}\left| g^{\varepsilon}_m\left( \alpha, {x}^{(m)}_i, r_{m}({x}^{(m)}_i)\right)\right| 
\\\nonumber &&\leq   \sum_{\begin{subarray}{c}{{m=0}}\\{m \, \mathrm{even}}\end{subarray}}^{n-2} \sum_{i=-\underline{c}_m}^{\overline{c}_m-1} \frac{r_m({x}^{(m)}_i)}{N} \sup_{\varepsilon <1}\left(
\varepsilon+ 
\left| \frac{\BS{\alpha}{{{x}^{(m)}_i}}{r_m({x}^{(m)}_i)}}{r_m({x}^{(m)}_i) } \right| \right) 
\\ && \leq \label{Gestimate}
 \sum_{\begin{subarray}{c}{{m=0}}\\{m \, \mathrm{even}}\end{subarray}}^{n-4} 
\frac{| P(C(n-m)^2)|}
{2^{n-m-4}}  +C'(C'+1)^2\left(1+M+\frac{1}{ q_{n-1} {m_{n-2}(x)}}\right),
\end{eqnarray}
where $P(x)$ is a fixed polynomial 
 and $m_{n-2}(x)= \min \{ T^i x, 1-T^ix; 0\leq i<N \} $. 
Let  $X_\nu = \{ x \st   {m_{n-2}(x)}q_{n-1} \leq \nu \}$. Hence, if $\alpha \notin A_N \cup A_N' $ and $x\notin X_{\nu }$, 
 since the series in (\ref{Gestimate}) is converging, $G_{\epsilon, \delta }(\alpha, x, N)$ is uniformly bounded by a constant  $T= T (C, C' , 1/\nu)$ which depends on $C, C' , 1/\nu$.
Since $X_\nu$ is such that $Leb(X_{\nu})\rightarrow 0 $ as $\nu\rightarrow 0$ uniformly in $ n$ (for $\alpha \notin A_{N'}$) and since moreover, from Lemma \ref{powerslemma} and from Theorem \ref{main},    $Leb(A_N(\delta))\rightarrow 0$ and $Leb(A_N'(\delta))\rightarrow 0$ as $\delta \rightarrow 0$, which correspond to choosing  $C'(\delta)$ and  $C(\delta)$ sufficiently large, 
this is enough to conclude the proof.
\end{proof}
Recall that fixed $\epsilon>0$, $ \delta>0$,  $G^{\epsilon,\delta}$ can be expressed as a function of finitely many random variables, listed in (\ref{Gdependence}). Some of them can be expressed through the random variables in (\ref{DN}, \ref{HN}, \ref{TN}) and all of them have a limiting distribution, either by  Corollary \ref{limitingquantities} or by Lemma \ref{limitingratios}.
\begin{lemma}\label{discontinuities}
For any $\epsilon>0$ and $\delta>0$, let $D$ be the set such that $G^{\epsilon,\delta}$ is discontinuous as a function of the random variables in (\ref{Gdependence}). Then the set of $(x,\alpha)$ such that the limits of the random variables in (\ref{Gdependence}) belong to $D$ has measure zero. 
\end{lemma}
\begin{proof}
Among the random variables in  (\ref{Gdependence}), expressed through (\ref{DN}, \ref{HN}, \ref{TN}), only two depend on $x$, i.e.~$D_N$ and $H_N$. 
 For each given value of all the other ones, one can see that there are only finitely many values of these two, near which $G^{\epsilon,\delta}$ changes discontinuously: more precisely, discontinuities might happen only when $d(x)$ and $h(x)$ correspond to discontinuities $x$ of the induced maps $T^m$ for $n-M\leq m \leq n-2$. Hence, the set $D$ has measure zero in the domain of $G^{\epsilon, \delta}$. Since $D_N$ and $H_N$ 
 are uniform random variables and have a uniformly distributed limit, also the set of $(x,\alpha)$ which are mapped to $D$ by their limit has measure zero.
\end{proof}

\begin{proofof}{Theorem}{limitingthm}
Let $f=f_1+f_2$ satisfy the assumptions of the Theorem. Since $f_2$ is integrable and for all $\alpha\in [0,1)\backslash \mathbb{Q}$ the rotation $R_{\alpha}$ is ergodic, by Birkhoff ergodic theorem, for a.e.~$(x,\alpha)$ the Birkhoff sums $S_N(\alpha, x,f_2)/N$ converge to the constant $\int f_2$. In particular, a.e.~convergence implies convergence in distribution. Let us hence consider separately the Birkhoff sums  $S_N(\alpha, x,f_1)$ of $f_1$.

To show that $S_N(x,\alpha, f_1)/N$ has a limiting distribution, it is enough to show that for each continuous and bounded function $g$, if $E$ denotes the expectation with respect to the Lebesgue measure on $(x,\alpha)$,
$\lim_{N\rightarrow \infty} Eg(S_N(x,\alpha, f_1)/N) =  Eg(S)$ 
for some random variable $S$.

Let us first show that, for each  $\epsilon>0,\delta>0$, $G^{\epsilon,\delta}(\cdot, \cdot, N)$ has a limiting distribution as $N$ tends to infinity. By Proposition \ref{generalBSprop},  $G^{\epsilon,\delta}$ can be expressed as a function of variables which,  by Corollary \ref{limitingquantities} and by Lemma \ref{limitingratios} have all a limiting distribution as $N$ tends to infinity. The condition on the discontinuity sets proved in Lemma \ref{discontinuities} is exactly what guarantees, by a standard result (see e.g. Theorem 2.1, Chapter III \S8 in \cite{Shy:pro}),  that also  $G^{\epsilon,\delta}$ has a limiting distribution. 

Hence, for some random variable $S^{\epsilon, \delta}$, $\lim_{N\rightarrow \infty} E g(G^{\epsilon, \delta}(\cdot, \cdot, N)) =  E g(S^{\epsilon, \delta})$ for each choice of $\epsilon$ and $\delta$, where  $g$ is as before any bounded and continuous function.
Using the tightness in Lemma \ref{tightness} and Prokhorov's theorem, one can show that there exists a  subsequence $S^{\epsilon_k,\delta_k}$ which converge in distribution to some $S$.

Let us prove the convergence in distribution of $S_N/N$. 
Given $g$ bounded and continuous and $\varepsilon >0 $, by the previous paragraph we can choose $k_0$ sufficiently large so that 
$| Eg(S^{\epsilon_{k_0},\delta_{k_0}}) -  Eg(S) | \leq \varepsilon/3$. We can estimate
\begin{eqnarray}
&&\hspace{-5mm} \left| E g\left(\frac{S_N(\cdot,\cdot, f_1)}{N}\right)  - Eg(S)\right| \leq  E \left|g\left(\frac{S_N(\cdot,\cdot, f_1)}{N}\right) -g(G^{\epsilon_{k_0},\delta_{k_0}}(\cdot,\cdot, N))\right| +\label{1final}\\ && \qquad \label{23final}  +  \left|E  g\left(G^{\epsilon_{k_0},\delta_{k_0}}(\cdot, \cdot, N)\right) -Eg\left(S^{\epsilon_{k_0},\delta_{k_0}}\right)\right| +
| Eg(S^{\epsilon_{k_0},\delta_{k_0}}) -  Eg(S) |.
\end{eqnarray}

By choice of $k_0$, the last term in (\ref{23final}) is less than $\varepsilon/3$. By the previous arguments, there exists some $N_1>0$, such that for all  $N\geq N_1$  also   the second term in (\ref{23final})  is less than $\varepsilon/3$. Substituting $k_0$ with a bigger one if necessary, we can assume that $\delta_{k_0} \leq \varepsilon / 12 \max g $. Moreover, if $k_0$ is large enough, using  
 absolute continuity of $g$ on a compact set given by tightness, by Proposition \ref{generalBSprop} there exists $N_2$ such that if $N\geq N_2$, the RHS term in (\ref{1final}) is controlled by $\delta_{k_0} 2\max g $ $+ \varepsilon/6$. Hence for each $N\geq \max(N_1, N_2)$, the LHS of  (\ref{1final}) is less than $\varepsilon$. This concludes the proof.
\end{proofof}

\appendix
\section{}
\subsection{Proof of Lemma \ref{powerslemma}.}\label{proofpowerslemma}
We present here the proof of Lemma \ref{powerslemma}, which is based on the techniques and results  used in \cite{SU:ren}. We just briefly recall the notation referring to the paper  \cite{SU:ren}  for further details.

As in \cite{SU:ren}, let  $\alpha\mapsto \G(\alpha)= \{ \frac{1}{\alpha} \}$ be the Gauss map  and $\mu_1$ its invariant measure given by the density $
 \frac{\ud \mu_1 }{\ud \alpha}=
(\ln 2 (1+\alpha))^{-1}$. Let $\h{\G}$ be the  natural extension of  $\G$, which acts as a shift on bi-infinite sequences $\h{\alpha} = \{a_n \}_{n\in \mathbb{Z}}= (\h{\alpha}^-, \h{\alpha}^+)$ and  let $\mu_2$ be the natural invariant measure for $\hat{\G}$, which satisfies $\mu_2 = \pi_* \mu_1$ where $\pi: \h{\alpha} = \{a_n \}_{n\in \mathbb{Z}} \mapsto \alpha = \h{\alpha}^+ = \{a_n \}_{n\in \mathbb{N}} $ is the natural projection.
Let $\{\Phi_t\}_{t\in \mathbb{R}}$ denote the special flow built over $\h{\G}$ under  the roof function $\varphi(\h{\alpha})= - \ln (\h{\G}(\h{\alpha})^-)$ and let $\mu_3$ denote the measure given by $\ud \mu_3 = \ud \mu_2\ud z$ on the domain $D$ of $\{\Phi_t\}_{t\in \mathbb{R}}$. As shown in \S4 in \cite{SU:ren}, $\{ \Phi_t\}_{t\in \mathbb{R}}$ is mixing.

\begin{proofof}{Lemma}{powerslemma}
Let $\h{A}_k= \{ \h{\alpha} \st a_{-k } \leq (k+1)^2\}$. It is easy to check that $\sum_{k=0}^{\infty} \mu_2(\h{A}_k^C) =\sum_{k=0}^{\infty}   \mu_1 \{ {\alpha} \st a_{1} \geq (k+1)^2 \} ) < \infty$.  Hence, by Borel-Cantelli Lemma, for a.e.~$\h{\alpha}$ there exists  $c=c(\h{\alpha}) >0$ such that $a_{-k }\leq c (k+1)^2$ for all $k\in \mathbb{N}$. Moreover, given $\delta'$, we can find $\h{A}= \h{A} (\delta')$ and $C=C(\delta')$ such that uniformly, for each $\h{\alpha} \notin  \h{A} $ and $k\in \mathbb{N}$,  $a_{-k }\leq C (k+1)^2$.

Since the condition (\ref{apowers}) depends only on $a_n$ with $n> 0$, i.e.~it is invariant on fibers $\pi^{-1}(\alpha)$, for any $c>0$, setting $n(N,\h{\alpha}) =n(N,\pi \h{\alpha}) $, we have
\begin{eqnarray}
 \mu_1 (\{ \alpha &:& a_{n(N,\alpha)-k}\leq  c (k+1)^2 , \qquad  0\leq k < n(N,\alpha) \}  = \nonumber \\
&=& \mu_2 (\{ \h{\alpha} \st a_{n(N,\h{\alpha})-k} \leq c (k+1)^2 , \qquad   0\leq k < n(N,\alpha) \} \geq \nonumber \\
&\geq & \mu_2 (\{ \h{\alpha} \st a_{n(N,\h{\alpha})-k} \leq c (k+1)^2 , \qquad   k < n(N,\alpha),\,\, k \in \mathbb{Z} \},  
 \label{manyalphah}
\end{eqnarray}
where the last inequality follows from the inclusions between the two sets. In order to conclude the proof, we want to show that for some $c>0$, (\ref{manyalphah}) is bigger than $1\!-\!\delta$ for all $N$ sufficiently large. To prove it, we will use mixing of the special flow $\{\Phi_t\}_{t\in \mathbb{R}}$. 

The set in  (\ref{manyalphah}) contains  $\{ \h{\alpha} \st \h{\G}^{n(N,\h{\alpha})}(\h{\alpha}) \notin \h{A}(\delta')\}$  if we take $c=C(\delta')$. 
 Let us localize the set of $\h{\alpha}$ considered so that we describe the set through the flow $\{\Phi_t\}_{t\in\mathbb{R}}$.  Reasoning as in the proof of Theorem 1.1 in \cite{SU:ren}, for each $\varepsilon>0$ we can construct a finite union of cylinders $\mathcal{C}$ and subsets $U_{\mathcal{C}}\subset \mathcal{C}$ such that
 $\sum_{\mathcal{C}}{\mu_2( U_{\mathcal{C}} )} \geq 1-\varepsilon$ and for all $\h{\alpha } \in U_{\mathcal{C}}$, if we set\footnote{We recall that  $f_{\mathcal{C}} = \sup_{\h{\alpha}\in \mathcal{C}}f(\h{\alpha})$ where $f=\lim_n f_n$ and $f_n(\h{\alpha}) = \ln q_{2n} (\h{\alpha}) - \BS{\varphi}{n}{\h{\alpha}}$
and we refer to \cite{SU:ren} for more comments of  $f$ and the proof of the existence of the limit $f$.}
$t_{\mathcal{C}} (N)= \ln N -  f_{\mathcal{C}}$ and denote by $p$ the projection $p(\h{\alpha},z)=\h{\alpha}$, 
we have \bes
\h{\G}^{n(N,\h{\alpha})}(\h{\alpha}) = p \Phi_{t_{\mathcal{C}}(N)} (\h{\alpha}, 0).
\ees
Let us denote by $A_{\Phi} = p^{-1} \h{A}$, so that $ \h{\G}^{n(N,\h{\alpha})}(\h{\alpha}) \in \h{A} $ iff $\Phi_{t_{\mathcal{C}}(N)}( \h{\alpha}, 0)  \in A_{\Phi} $.  For each $\varepsilon>0$, by absolute continuity of the integral, there exists $\delta'>0$ and $\h{A}=\h{A}(\delta') $ so that   $\mu_3(A_{\Phi})\leq \varepsilon$.
Hence, thickening $U_{\mathcal{C}}$  slightly, i.~e.~considering $U_{\mathcal{C}}^{\delta_{\mathcal{C}}} =U_{\mathcal{C}} \times[0, \delta_{\mathcal{C}})$ for some small $\delta_{\mathcal{C}}$,
  as in the proof of  \cite{SU:ren}, and exploiting mixing of the flow $\{\Phi_{t}\}_{t\in \mathbb{R}}$, there exists $N_{\mathcal{C}} $ so that for each $N\geq N_{\mathcal{C}} $
\bes
\begin{split}
\mu_{2}\{ \h{\alpha} \in {\mathcal{C}} \st \h{\G}^{n(N,\h{\alpha})}(\h{\alpha}) \notin \h{A}\} &\geq \frac{  \mu_3( U_{\mathcal{C}}^{\delta_{\mathcal{C}}} \cap\Phi_{-t_{\mathcal{C}}(N)}(  A_{\Phi}^C))}{\delta_{\mathcal{C}}} \\ 
& \geq (1-\varepsilon)(1-\mu_3(A_{\Phi}))\mu_2(U_\mathcal{C}).
\end{split}
\ees
Summing over the finitely many $\mathcal{C}$ involved and choosing $\varepsilon$ so that $(1-\varepsilon)^3\geq (1-\delta)$ and ${N_0} = \max  _{\mathcal{C}} N_{\mathcal{C}} $, this concludes the proof.
\end{proofof}

\subsection{On the proof of Theorem \ref{main}.}\label{appendixthmdiff}
Let us sketch briefly  how to obtain Theorem \ref{main} from 
Theorem $1$, \cite{SU:ren}.
The only differences between the two theorems 
 are the following. First the measure considered is the $2$-dimensional Lebesgue measure $Leb$, while Theorem  $1$ in \cite{SU:ren} is stated for the product $\G \times \lambda$ of the Gauss measure and $\lambda$, Lebesgue measure. Then the entries considered are $a_{n(N)+k}$ for $| k | \leq  M$  instead than only $a_{n(N)+k}$ for $0< k \leq  M$. Furthermore, $n$ is required to be even.

The first two differences require  easy modifications. It was already remarked in  \cite{SU:ren} that  Theorem $1$, \cite{SU:ren} holds for any absolutely continuous measure. To consider also $a_{n(N)+k}$ with $-M\leq k
\leq 0$ it is enough to substitute the cylinder $C_N$ in ($24$),\cite{SU:ren}  with
$ C_M := \h{\G}^{M-1} \hC{ c_{M} c_{M-1} \dots, c_0, c_{-1}, \dots,  c_{-M} }$ (we remark that here $M$ plays the role of $N$ in \cite{SU:ren}).

In order to have a limiting distributions when only even $n$ are considered (a choice which simplifies our analysis), it is necessary to use a slightly different special flow than the one in $\S 2$ of \cite{SU:ren}. Instead than the base transformation  $\h{\G}$, consider the transformation  $\h{\G^2}$ and substitute the roof function $\varphi $ in $(6)$ of \cite{SU:ren} with
\bes
\varphi(\h{\alpha}) = \ln \left( a_1 + \frac{1}{a_0 + \frac{1}{a_{-1} + {\dots }}} \right) + \ln \left( a_2 + \frac{1}{a_1 + \frac{1}{a_{0} + {\dots }}} \right) .
\ees
In this way, Lemma $1$ of \cite{SU:ren} holds for $f_n(\h{\alpha}) = \ln q_{2n} (\h{\alpha}) - \BS{\varphi}{\h{\alpha}}{n}$. The proof that the suspension flow under this new $\varphi$ is mixing proceeds as in $\S 4$ in \cite{SU:ren}: one can explicitly write the equations of local stable and unstable manifolds and check that they are non-integrable.


\bibliography{bibliography}

\bibliographystyle{amsalpha}


\end{document}